\documentclass[11pt]{amsart}
\usepackage[utf8]{inputenc}
\usepackage{amsmath,amssymb,amsthm,dsfont,mathrsfs}
\usepackage{xcolor}
\usepackage[colorlinks=true,linkcolor=blue,citecolor=blue]{hyperref}

\allowdisplaybreaks

\numberwithin{equation}{section}
\newtheorem{theorem}{Theorem}[section]
\newtheorem{corollary}[theorem]{Corollary}
\newtheorem{proposition}[theorem]{Proposition}

\newtheorem{lemma}[theorem]{Lemma}

\theoremstyle{definition}
\newtheorem{definition}[theorem]{Definition}
\newtheorem{remark}[theorem]{Remark}
\newtheorem{example}[theorem]{Example}

\DeclareMathOperator{\R}{\mathbb{R}}
\DeclareMathOperator{\Z}{\mathbb{Z}}
\DeclareMathOperator{\N}{\mathbb{N}}

\DeclareMathOperator{\T}{\mathbb{T}}

\DeclareMathOperator{\pt}{\mathnormal{\prescript{t}{}{p(D)}}}

\usepackage{mathtools}

\title[Global Hypoellipticity and Solvability]{Global Hypoellipticity and Solvability \\ with Loss of Derivatives on the Torus}
\author{André Kowacs} 
 \address{Universidade Federal do Paran\'{a}, 
	Departamento de Matem\'{a}tica,
	 CEP 81531-990, Curitiba, Brazil}
\email{andrekowacs@gmail.com}

   \endgraf
 \author{Alexandre Kirilov} 
 \address{Universidade Federal do Paran\'{a}, 
	Departamento de Matem\'{a}tica,
	C.P.19096, CEP 81531-990, Curitiba, Brazil}
\email{akirilov@ufpr.br}

\thanks{The second author was supported in part by CNPq (grants 316850/2021-7 and 423458/2021-3). }


\begin{document}

	\begin{abstract}
		This paper provides a complete characterization of global hypoellipticity and solvability with loss of derivatives for Fourier multiplier operators on the $n$-dimensional torus. We establish necessary and sufficient conditions for these properties and examine their connections with classical notions of global hypoellipticity and solvability, particularly in relation to the closedness of the operator's range.
		
		As an application, we explore the interplay between these properties and number theory in the context of differential operators on the two-torus. Specifically, we prove that the loss of derivatives in the solvability of the vector field $\partial_{x_1} - \alpha \partial_{x_2}$ is precisely determined by the well-known irrationality measure $\mu(\alpha)$ of its coefficient $\alpha$.
		Furthermore, we analyze the wave operator $\partial_{x_1}^2 - \eta^2 \Delta_{\mathbb{T}^n}$ and show how the loss of derivatives depends explicitly on the parameter $\eta > 0$.
	\end{abstract}

\maketitle
 \section{Introduction}
 	
 	The study of global properties of differential operators on compact manifolds has been an active area of research in recent decades. A fundamental question in this field is to  determine the conditions under which an operator fails to regularize distributions —a property known as global hypoellipticity— or has closed range, which is often associated with global solvability. These questions lie at the heart of the theory of linear partial differential equations and reveal deep connections with number theory and geometry, as extensively discussed in
 	\cite{
 		AP2006_ferrara,
 		AFR2022_jam, 
 		Berg1999_tams,
 		CC2000_cpde,
 		Avil2023_mn,
 		Petr2011_tams}  
 	and the references therein.
 	
 	In addition to classical notions of smoothness, variations of global hypoellipticity and solvability have been investigated by considering alternative forms of regularity, such as analyticity and ultradifferentiability. These extensions have been explored in works such as  
 	\cite{
 		AKM2019_jmaa,
 		BDG2018_jde,
 		AM2022_adm,
 		Les2021_afm,
 		Petr2009_mn}.
 	 
 	 In this work, we adopt an alternative approach to hypoellipticity and solvability by focusing on spaces of lower regularity and explicitly accounting for the loss of derivatives (regularity). While this type of global hypoellipticity has been explored in other works, such as \cite{Ferra2023,Ferra2020,FerraPetronilho2021}, our approach stands out by enabling a precise quantification of the loss of regularity.
 	 
 	 Similarly, in the context of solvability, previous studies have addressed lower regularity settings, such as in \cite{CORDARO2016458}, where the authors investigate semi-global solvability with loss of regularity over Sobolev spaces on general smooth manifolds. However, our work differs substantially due to its fully global perspective, providing a comprehensive analysis of these properties on compact manifolds.

 	 We introduce the notions of global hypoellipticity and solvability with a finite loss of derivatives and establish necessary and sufficient conditions for a Fourier multiplier operator on the torus to satisfy these properties. In particular, we show that a Fourier multiplier on the \( n \)-torus is globally hypoelliptic if and only if it is globally hypoelliptic with loss of \( r \) derivatives for some \( r \geq 0 \). 
 	 
 	 One of the main results of this paper is the following characterization of global hypoellipticity with a loss of $ r $ derivatives:
 	 
 	 \begin{theorem}
 	 	Let $ p(D) $ be a Fourier multiplier of order $ m $ on $ \mathbb{T}^n $, and let $ r \geq 0 $. The following conditions are equivalent:
 	 	\begin{enumerate}
 	 		\item The symbol $ p(\xi) $ vanishes for at most finitely many $ \xi \in \mathbb{Z}^n $, and there exists a constant $ C > 0 $ such that  
 	 		\begin{equation*}
 	 		    |p(\xi)| \geq C |\xi|^{m - r},
 	 		\end{equation*} 		
 	 		for every $ \xi \in \mathbb{Z}^n \setminus \{0\} $ such that $ p(\xi) \neq 0 $.
 	 		
 	 		\item For every $ k \in \mathbb{R} $, the following implication holds:
 	 		\[
 	 		u \in \mathscr{D}'(\mathbb{T}^n) \text{ and } p(D)u \in H^k(\mathbb{T}^n) \implies u \in H^{k + m - r}(\mathbb{T}^n).
 	 		\]
 	 	\end{enumerate}
 	 	
 	 	Moreover, $ p(D) $ is globally hypoelliptic in the classical sense if and only if either of the equivalent conditions above holds for some $ r \geq 0 $.
 	 \end{theorem}
 	 
 	 Similarly, we establish the following characterization of global solvability with a loss of $ r $ derivatives:
 	 
 	 \begin{theorem}
 	 	Let $ p(D) $ be a Fourier multiplier of order $ m $ on $ \mathbb{T}^n $, and let $ r \geq 0 $. The following conditions are equivalent:
 	 	\begin{enumerate}
 	 		\item There exists a constant $ K > 0 $ such that  
 	 		\[
 	 		|p(\xi)| \geq K |\xi|^{m - r},
 	 		\]
 	 		for every $ \xi \in \mathbb{Z}^{n} \setminus \{0\} $ such that $ p(\xi) \neq 0 $.
 	 		
 	 		\item For every $ k \in \mathbb{R} $ and for every $ f \in H^k(\mathbb{T}^n) \cap (\ker {}^{t}p(D) \cap C^\infty(\mathbb{T}^n))^0 $, there exists $ u \in H^{k + m - r}(\mathbb{T}^n) $ such that  
 	 		\[
 	 		p(D)u = f.
 	 		\]
 	 		
 	 		\item For every $ k \in \mathbb{R} $, the operator  
 	 		\[
 	 		p(D): H^{k + m - r}(\mathbb{T}^n) \to H^k(\mathbb{T}^n)
 	 		\]
 	 		has closed range.
 	 	\end{enumerate}
 	 \end{theorem}
 	 
 	 As an application, we revisit the classical work of Greenfield and Wallach \cite{GW1972_pams} on the interplay between global hypoellipticity and Liouville numbers. In this context, we establish deeper and more refined connections between number theory and the regularity of vector fields on the two-dimensional torus.  
 	 
 	 \begin{theorem}\label{Thm-GHr-introduction}  
 	 	Let \( p(D) = \partial_{x_1} - \alpha \partial_{x_2} \) be a vector field on \( \mathbb{T}^2 \), where \( \alpha \in \mathbb{C} \). The global hypoellipticity with loss of derivatives for \( p(D) \) is characterized as follows:  
 	 	\begin{enumerate}  
 	 		\item[{\it (i)}] If \( \operatorname{Im}(\alpha) \neq 0 \), then \( p(D) \) is globally hypoelliptic without any loss of derivatives.  
 	 		
 	 		\item[{\it (ii)}] If \( \alpha \) is either a rational number or an irrational Liouville number, then \( p(D) \) is not globally hypoelliptic with loss of derivatives.  
 	 		
 	 		\item[{\it (iii)}] If \( \alpha \) is an algebraic irrational number, then \( p(D) \) is globally hypoelliptic with a loss of $r$ derivatives, for any $r>2$. 
 	 		
 	 		\item[{\it (iv)}] If \( \alpha \) is a transcendental irrational number with a finite irrationality measure \( \mu(\alpha) \), then \( p(D) \) is globally hypoelliptic with a loss of $r$ derivatives, for any $r>\mu(\alpha)$. 
 	 	\end{enumerate}  
 	 \end{theorem}  
 	 
 	 Here, \( \mu(\alpha) \) denotes the Liouville–Roth exponent of \( \alpha\) (also referred to as the irrationality measure of of \(alpha\)), defined as the infimum of all real numbers \( \mu > 0 \) for which the inequality  
 	 \[
 	 0 < \left| \alpha - \frac{p}{q} \right| < \frac{1}{q^\mu}
 	 \]
 	 has at most finitely many rational solutions \( p/q \). This definition extends the concept of Liouville numbers by quantifying how well \( \alpha \) can be approximated by rationals.

 	 The paper is organized as follows. In Section 2, we introduce the notion of global hypoellipticity with a loss of derivatives and establish necessary and sufficient conditions for this property. In Section 3, we define and characterize global solvability with a loss of derivatives, proving its equivalence with the closed range condition. Section 4 focuses on applications to vector fields on \( \mathbb{T}^2 \), where we compute explicit indices and explore their connections with Diophantine approximation. Finally, in Section 5, we extend our analysis to the wave operator on the torus, utilizing several number-theoretic tools to study its behavior.

\section{Hypoellipticity with Loss of Derivatives}\label{Section_GH-r}

	Let \( \mathscr{D}'(\mathbb{T}^n) \) denote the space of distributions on the \( n \)-dimensional torus \( \mathbb{T}^n \), and consider the classical Sobolev scale of spaces \( H^k(\mathbb{T}^n) \) for \( k \in \mathbb{R} \). These are Banach spaces consisting of all distributions \( u \in \mathscr{D}'(\mathbb{T}^n) \) for which the norm  
	\[
	\|u\|_{H^k(\mathbb{T}^n)}^2 = \sum_{\xi \in \mathbb{Z}^n} (1 + \| \xi \|^2)^{k} |\widehat{u}(\xi)|^2 < +\infty.
	\]
	
	We denote  \( |\xi| = \sum_j |\xi_j| \)  and \( \|\xi\| = \left( \sum_j |\xi_j|^2 \right)^{1/2}, \)  for \( \xi \in \mathbb{Z}^n\).
		
	We recall that a function \( p: \mathbb{T}^n \times \mathbb{Z}^n \to \mathbb{C} \) is called a symbol of order \( m \in \mathbb{R} \), written \( p \in S^m(\mathbb{T}^n \times \mathbb{Z}^n) \), if it satisfies the estimate  
	\[
	|\partial_x^\alpha p(x, \xi)| \leq C_\alpha (1 + \|\xi\|^2)^{\frac{m}{2}}
	\]
	for some constant \( C_\alpha > 0 \), for all multi-indices \( \alpha \in \mathbb{N}_0^n \) and for all \( (x,\xi) \in \T^n\times\mathbb{Z}^n \).  
	
	Given such a symbol, we define the pseudo-differential operator \( p(D): \mathscr{D}'(\mathbb{T}^n) \to \mathscr{D}'(\mathbb{T}^n) \) by  
	\[
	p(D)u(x) = \sum_{\xi \in \mathbb{Z}^n} e^{i x \cdot \xi} p(x, \xi) \widehat{u}(\xi), \quad u \in \mathscr{D}'(\mathbb{T}^n).
	\]
	In this case, we say that \( p(D) \) is an operator of order \( m \), consistent with the order of its symbol.
	
	When the symbol \( p \) is independent of \( x \), we have  
	\[
	\widehat{p(D)u}(\xi) = p(\xi) \widehat{u}(\xi), \quad \xi \in \mathbb{Z}^n,
	\]
	and in this case, \( p(D) \) is called a Fourier multiplier.  
	
	For further details on pseudo-differential operators on the torus and classical results used in this manuscript, we refer the reader to the works of Ruzhansky and Turunen \cite{RT2007_fourier_torus, RT2010_book, RT2010_jfaa}, as well as the article by Ferra and Petronilho \cite{FerraPetronilho2021}.  
	
	Note that the order of a symbol or operator is not uniquely determined: if a symbol \( p \) is of order \( m \in \mathbb{R} \), then it is also of order \( m' \) for any \( m' \geq m \). To precisely characterize the regularity properties of the operators studied in this paper, we introduce the following refined notion of order.  
	
	\begin{definition}
		Let \( p: \mathbb{T}^n \times \mathbb{Z}^n \to \mathbb{C} \) be a symbol. We say that \( p \) has \emph{intrinsic order} \( m \) if  
		\[
		m = \inf \{ m' \in \mathbb{R} : p \text{ is of order } m' \},
		\]
		and we write \( p \in S^{[m]}(\mathbb{T}^n \times \mathbb{Z}^n) \).  
		
		If \( p \) is not of order \( m \) for any \( m \in \mathbb{R} \), we say that \( p \) has intrinsic order \( \infty \). Similarly, we define the intrinsic order for the associated operator.  
		
		In particular, we denote by \( S^{[m]}(\mathbb{Z}^n) \) the class of symbols of intrinsic order \( m \in \mathbb{R} \) that are independent of \( x \).
	\end{definition}
	
	For instance, a partial differential operator of order \( m \) has intrinsic order \( m \). Moreover, one can easily verify that an elliptic pseudo-differential operator of order \( m \) also has intrinsic order \( m \). On the other hand, the symbol  
	\[
	p(\xi) = \log(1+|\xi|), \quad \xi \in \mathbb{Z},
	\]
	is not of order \( 0 \), but it is of order \( \varepsilon \) for every \( \varepsilon > 0 \). Therefore, \( p \) has intrinsic order \( 0 \).
    
	A well-known result regarding pseudo-differential operators states that a symbol if a \( p \in S^{m}(\mathbb{T}^n \times \mathbb{Z}^n) \) then \( u \in H^{k+m}(\mathbb{T}^n) \implies  p(D)u \in H^k(\mathbb{T}^n) \) for all \( k \in \mathbb{R} \).  However, the reverse implication does not necessarily hold. 
	
	For instance, consider the differential operator \( \partial_{x_2} \) on \( \mathbb{T}^2 \) and the distribution \( u = \delta_0 \otimes 1 \), where \( \delta_0 \) is the Dirac delta distribution supported at \( x_1 = 0 \).  Although \( \partial_{x_2} \) is a first-order differential operator and \( u \notin  H^{1}(\T^2) \), applying \( \partial_{x_2} \) yields
	\begin{equation*}
		\partial_{x_2} u = 0 \in H^{0}(\T^2),
	\end{equation*}
	and, in fact, \( \partial_{x_2} u \in C^\infty(\mathbb{T}^2)=\bigcap_{k\in\R}H^{k}(\T^2) \).
	
	In this sense, the operator \( \partial_{x_2} \) exhibits a regularizing effect on certain distributions. The goal of this section is to investigate this phenomenon in greater depth.

	\begin{definition}
		Let \( p \in S^{[m]}(\mathbb{Z}^n) \). We say that the operator \( p(D) \) is \textit{globally hypoelliptic with loss of \( r \geq 0 \) derivatives} (denoted GH-\( r \)) if, for every \( k \in \mathbb{R} \), the following implication holds:
		\[
		u \in \mathscr{D}'(\mathbb{T}^n) \text{ and } p(D)u \in H^k(\mathbb{T}^n) \implies u \in H^{k + m - r}(\mathbb{T}^n).
		\]
		
		Additionally, we define the \textit{global hypoellipticity index} of \( p(D) \) as
		\[
		\operatorname{ind}_{GH}(p(D)) \coloneqq \inf\{r \geq 0 : p(D) \text{ is GH-}r\}.
		\]
		with the convention that \( \operatorname{ind}_{GH}(p(D)) = \infty \) if \( p(D) \) is not GH-\( r \) for any \( r \geq 0 \).
	\end{definition}
   
	\begin{remark}
		Note that in the definition above, if we replace the intrinsic order \( m \) of the operator \( p(D) \) with a different order \( m'  \geq m \), then for a given \( r \geq 0 \), the property of \( p(D) \) being GH-\( r \) may depend on the choice of \( m' \). Specifically, the condition \( |p(\xi)| \geq C |\xi|^{m' - r} \), given in the introduction, is sensitive to the value of \( m' \).
		
		Furthermore, by the Sobolev embedding theorem, if \( p(D) \) is GH-\( r \), then for any \( k \in \mathbb{N}_0 \) such that \( k > r+\frac{n}{2}-m \), the condition \( p(D)u = f \in C^k(\mathbb{T}^n) \) implies \( u \in C^{k'}(\mathbb{T}^n) \), where $k'$ denotes the greatest integer {\bf less than} $k+m-r-\frac{n}{2}$.  
        
	\end{remark}

	We now proceed to establish the main result of this section. Our objective is to characterize the global regularization property of finite order for an operator in terms of the asymptotic behavior of its symbol.

	\begin{theorem}\label{theoGH}
		Let \( r \geq 0 \) and \( p \in S^{[m]}(\mathbb{Z}^n) \). The operator \( p(D) \) is GH-\( r \) if and only if 
		the symbol \( p(\xi) \) vanishes for at most finitely many \( \xi \in \mathbb{Z}^n \) and there exists $K>0$ such that 
		\begin{equation}\label{ineq-GH}
			 |p(\xi)|\geq K|\xi|^{m-r},
		\end{equation}
		for every \( \xi \in \mathbb{Z}^n \setminus \{0\} \) such that \( p(\xi) \neq 0 \).
	\end{theorem}

\begin{proof}
	Suppose first that there exists a sequence \( (\xi_j)_{j\in\N} \) of distinct elements in \( \Z^n\setminus\{0\} \) such that \( p(\xi_j) = 0 \) for every \( j \in \mathbb{N} \). Define a sequence of Fourier coefficients by setting \( \widehat{u}(\xi_j) = 1 \) for all \( j \in \mathbb{N} \) and \( \widehat{u}(\xi) = 0 \) for every other \( \xi \in \mathbb{Z}^n \). Consider the distribution \( u \in \mathscr{D}'(\mathbb{T}^n) \) given by  
	\[
	u(x) = \sum_{\xi\in\Z^n} \widehat{u}(\xi) e^{i\xi\cdot x}.
	\]
	
	Since \( p(\xi_j) = 0 \), it follows that \( p(D)u = 0 \in H^{-m+r}(\mathbb{T}^n) \). However, we also have \( u \notin H^0(\mathbb{T}^n) \), and therefore \( p(D) \) is not GH-\( r \).
	
	\medskip
	Next, suppose that inequality \eqref{ineq-GH} does not hold. Then there exists a sequence \( (\xi_j)_{j\in\N} \) of distinct elements in \( \mathbb{Z}^n \setminus \{0\} \) such that  
	\[
	0 < |p(\xi_j)| \leq  \frac{1}{j} |\xi_j|^{m-r}, \quad j \in \mathbb{N}.
	\]
	
	Define \( \widehat{u}(\xi_j) = 1 \) for all \( j \in \mathbb{N} \) and \( \widehat{u}(\xi) = 0 \) for every other \( \xi \in \mathbb{Z}^n \). Consider again the distribution \( u \in \mathscr{D}'(\mathbb{T}^n) \) given by  
	\[
	u(x) = \sum_{\xi\in\Z^n} \widehat{u}(\xi) e^{i\xi\cdot x}.
	\]
	
	Then, for every \( j \in \mathbb{N} \), we obtain
	\begin{equation*}
		|\widehat{p(D)u}(\xi_j)| = |p(\xi_j)\widehat{u}(\xi_j)| \leq \frac{1}{j} |\xi_j|^{m-r},
	\end{equation*}
	while for all other \( \xi \in \mathbb{Z}^n \setminus \{0\} \), we have \( \widehat{p(D)u}(\xi) = 0 \). This shows that \( p(D)u \in H^{-m+r}(\mathbb{T}^n) \). Indeed, we compute:
	\begin{align*}
		\|p(D)u\|_{H^{-m+r}(\mathbb{T}^n)}^2
		&= \sum_{\xi \in \mathbb{Z}^n} (1+\|\xi\|^2)^{-m+r} |\widehat{p(D)u}(\xi)|^2 \\
		&= \sum_{j \in \mathbb{N}} (1+\|\xi_j\|^2)^{-m+r} \frac{1}{j^2} |\xi_j|^{2(m-r)} \\
		&\leq (1+n)^{|m-r|} \sum_{j \in \mathbb{N}} \frac{1}{j^2} < \infty,
	\end{align*}
	where, in the last inequality, we used the bound  
	\[
	|\xi|^{-2\tau} \leq (1+n)^{|\tau|} (1+\|\xi\|^2)^{-\tau}, \quad \xi \in \mathbb{Z}^n \setminus \{0\},  \tau \in \mathbb{R}.
	\]
	
	This proves that \( p(D)u \in H^{-m+r}(\mathbb{T}^n) \). However, we also have  
	\begin{align*}
		\|u\|_{H^0(\mathbb{T}^n)}^2 &= \sum_{j=1}^{\infty} 1 = \infty,
	\end{align*}
	which implies that \( u \notin H^0(\mathbb{T}^n) \). Consequently, \( p(D) \) is not GH-\( r \).

	\medskip
	On the other hand, assume that both conditions in the theorem hold for some constant \( K > 0 \).
	
	Let \( u \in \mathscr{D}'(\mathbb{T}^n) \) be such that \( p(D)u = f \in H^k(\mathbb{T}^n) \), for some \( k \in \mathbb{R} \). Then, for all but finitely many \( \xi \in \mathbb{Z}^n \), we have  
	\[
	\widehat{u}(\xi) = \frac{\widehat{f}(\xi)}{p(\xi)}.
	\]
	Using this expression, we estimate the Sobolev norm of \( u \):  
	\begin{align*}
		\|u\|_{H^{k+m-r}(\mathbb{T}^n)}^2 &= \sum_{\xi \in \mathbb{Z}^n} (1+\|\xi\|^2)^{k+m-r} |\widehat{u}(\xi)|^2 \\
		&= \sum_{p(\xi) = 0} (1+\|\xi\|^2)^{k+m-r} |\widehat{u}(\xi)|^2 \\
		& \quad + \sum_{p(\xi) \neq 0} (1+\|\xi\|^2)^{k+m-r} \frac{|\widehat{f}(\xi)|^2}{|p(\xi)|^2}.
	\end{align*}
	
	Since \( p(\xi) = 0 \) for only finitely many \( \xi \), the first sum is finite and can be bounded by some constant \( K' > 0 \). For the second sum, applying the lower bound assumption on \( |p(\xi)| \), we obtain  
	\begin{align*}
		\sum_{p(\xi) \neq 0} (1+\|\xi\|^2)^{k+m-r} \frac{|\widehat{f}(\xi)|^2}{|p(\xi)|^2}
		&\leq \frac{(1+n)^{|m-r|}}{K^2}  \sum_{p(\xi)\neq 0} (1+\|\xi\|^2)^k |\widehat{f}(\xi)|^2 \\
		&\leq  \frac{(1+n)^{|m-r|}}{K^2} \|f\|_{H^k(\mathbb{T}^n)}^2 < \infty.
	\end{align*}

	Thus, \( \|u\|_{H^{k+m-r}(\mathbb{T}^n)} < \infty \), which implies that \( u \in H^{k+m-r}(\mathbb{T}^n) \), as required. 	
	We conclude that \( p(D) \) is GH-\( r \), completing the proof.
	\end{proof}
	
	The proof of Theorem \ref{theoGH} naturally leads to the following result, which provides a lower bound for the global regularization index of an operator.
	
	\begin{corollary}\label{coroGH}
		Let \( r' > 0 \) and \( p \in S^{[m]}(\mathbb{Z}^n) \). 
		If there exists a constant \( K' > 0 \) such that
		\begin{equation}\label{ineq-not-GH-coro}
			|p(\xi)| \leq K'|\xi|^{m - r'}
		\end{equation}
		for infinitely many \( \xi \in \mathbb{Z}^n \setminus \{0\} \), then the operator \( p(D) \) is not GH-\( r \) for any \( 0 \leq r < r' \). 
		Consequently, \[ \operatorname{ind}_{GH}(p(D)) \geq r'. \]
	\end{corollary}
	
	\begin{proof}
		If \( p(\xi) \) has infinitely many zeros, the previous theorem ensures that the operator is not GH-\( r \). Thus, we may assume that the set of zeros of \( p(\xi) \) is finite and that \( p(\xi) \) satisfies \eqref{ineq-not-GH-coro}. 
		
		Let \( r < r' \) and write \( r' = r + \varepsilon \) for some \( \varepsilon > 0 \). Under this assumption, we obtain the bound  
		\begin{equation*}
			0 < |p(\xi)| \leq K' |\xi|^{m - r} |\xi|^{-\varepsilon},
		\end{equation*}
		for infinitely many \( \xi \in \mathbb{Z}^n \setminus \{0\} \).
				
		By selecting an appropriate sequence, we may assume that \( |\xi_j| \geq (K' j)^{\frac{1}{\varepsilon}} \) for every \( j \in \mathbb{N} \). This yields  
		\begin{equation*}
			|p(\xi_j)| \leq \frac{1}{j} |\xi_j|^{m - r},
		\end{equation*}
		for all \( j \in \mathbb{N} \). 
		
		As in the proof of Theorem \ref{theoGH}, this estimate ensures that \( p(D) \) fails to be GH-\( r \), as claimed.
	\end{proof}
	
	The proof of the last corollary provides a natural framework for a better understanding of the global hypoellipticity index. Let us establish the precise conditions on the symbol that characterize this index.
	
	\begin{theorem}\label{theoGHindex}
		Let \( r \geq 0 \) and \( p \in S^{[m]}(\mathbb{Z}^n) \). Then  \( \operatorname{ind}_{GH}(p(D)) = r \) if and only if the following conditions hold:
		\begin{enumerate}
			\item[{\it (i)}] The symbol \( p(\xi) \) vanishes for only finitely many \( \xi \in \mathbb{Z}^n \).
			
			\item[{\it (ii)}] For every \( \varepsilon > 0 \), there exists a constant \( K_\varepsilon > 0 \) such that
			\begin{equation}\label{ineqGH}
				|p(\xi)| \geq K_\varepsilon |\xi|^{m - (r + \varepsilon)},
			\end{equation}
			for all \( \xi \in \mathbb{Z}^n \setminus \{0\} \) satisfying \( p(\xi) \neq 0 \).
			
			\item[{\it (iii)}] If \( r > 0 \), then for every \( \varepsilon > 0 \), there exists a constant \( K_{\varepsilon}' > 0 \) such that
			\begin{equation}\label{ineqGH_2}
				0 < |p(\xi)| \leq K_{\varepsilon}' |\xi|^{m - (r - \varepsilon)},
			\end{equation}
			for infinitely many \( \xi \in \mathbb{Z}^n \setminus \{0\} \).
		\end{enumerate}
	\end{theorem}
	
	\begin{proof}
		Assume first that conditions {\it (i)}, {\it (ii)}, and {\it (iii)} hold. By Theorem \ref{theoGH}, conditions {\it (i)} and {\it (ii)} ensure that \( p(D) \) is GH-\( (r + \varepsilon) \) for every \( \varepsilon > 0 \), which implies \( \operatorname{ind}_{GH}(p(D)) \leq r \).  
		
		If \( r = 0 \), we conclude that \( \operatorname{ind}_{GH}(p(D)) = 0 \). If \( r > 0 \), Corollary \ref{coroGH} and condition {\it (iii)} imply that \( p(D) \) is not GH-\( s \) for any \( s < r \). Hence, we obtain \( \operatorname{ind}_{GH}(p(D)) = r \), as claimed.  
		
		Conversely,  if {\it (i)} or {\it (ii)} fail, then by Theorem \ref{theoGH}, there exists \( \varepsilon > 0 \) such that \( p(D) \) is not GH-\( (r+\varepsilon) \), leading to \( \operatorname{ind}_{GH}(p(D)) > r \).  
			
		If {\it (iii)} fails, then there exists \( \varepsilon > 0 \) such that  
			\begin{equation*}
				0 < |p(\xi)| \leq |\xi|^{m - (r - \varepsilon)}
			\end{equation*}
			only for finitely many \( \xi \in \mathbb{Z}^n \setminus \{0\} \). In this case, we can choose \( K > 0 \) sufficiently large so that  
			\begin{equation*}
				|p(\xi)| \geq K |\xi|^{m - (r - \varepsilon)},
			\end{equation*}
			for all \( \xi \in \mathbb{Z}^n \setminus \{0\} \) with \( p(\xi) \neq 0 \). We may also assume by the previous case that $p(\xi)=0$ for only finitely many $\xi\in\Z^n$. 
			
			By Theorem \ref{theoGH}, this implies that \( p(D) \) is GH-\( (r-\varepsilon) \),  thus, we must have \( \operatorname{ind}_{GH}(p(D)) < r \), completing the proof.
	\end{proof}
	
	\begin{remark}\label{elliptic_is_GH-0}
		A direct consequence of Theorem \ref{theoGH} is that every elliptic Fourier multiplier is GH-\( 0 \). Indeed, an elliptic symbol of intrinsic order \( m \) is also of order \( m \), and such operators satisfy the conditions of Theorem \ref{theoGH}. Consequently, their global hypoellipticity index is zero. However, not every operator with a global hypoellipticity index equal to zero is elliptic, as illustrated by the following example.
	\end{remark}
	
	\begin{example}\label{example_log}
		Consider the Fourier multiplier on \( \mathbb{T}^1 \) with symbol  
		\[
		p(\xi) = \frac{\xi}{\log(e + |\xi|)}, \quad \xi \in \mathbb{Z}.
		\]
		
		It is clear that \( p(\xi) \) has intrinsic order \( 1 \). Moreover, for every \( \varepsilon > 0 \), there exists a constant \( K_{\varepsilon} > 0 \) such that  
		\begin{equation*}
			|p(\xi)| \geq K_{\varepsilon} |\xi|^{1 - \varepsilon}, \quad \xi \in \mathbb{Z} \setminus \{0\}.
		\end{equation*}
		
		This follows from the asymptotic behavior  
		\[
		\lim_{|\xi| \to \infty} \frac{|\xi|^{\varepsilon}}{\log(e + |\xi|)} = \infty.
		\]
		
		Thus, by Theorem \ref{theoGHindex}, we conclude that  
		\[
		\operatorname{ind}_{GH}(p(D)) = 0.
		\]

		However, observe that there does not exist a constant \( K > 0 \) such that  
		\begin{equation*}
			|p(\xi)| \geq K |\xi|, \quad \xi \in \mathbb{Z} \setminus \{0\}.
		\end{equation*}
		
		This follows from the fact that  
		\[
		\lim_{|\xi| \to \infty} \frac{|p(\xi)|}{|\xi|} = 0.
		\]
		
		Therefore, \( p(D) \) is not elliptic, confirming that the converse of Remark \ref{elliptic_is_GH-0} does not hold.
	\end{example}
	
	\begin{example}\label{Laplacian_GH-0}
		The Laplacian \( \Delta_{\mathbb{T}^n} \), whose symbol is given by 
		\[
		p(\xi) = -\sum_{j=1}^{n} \xi_j^2,
		\]
		satisfies the conditions of Theorem \ref{theoGH} with \( r = 0 \). Consequently, \( \Delta_{\mathbb{T}^n} \) is globally hypoelliptic with no loss of derivatives, i.e., GH-\( 0 \).
		
		Similarly, the operator \( \partial_x \) on \( \mathbb{T}^1 \), with symbol \( p(\xi) = i\xi \), also satisfies these conditions and is therefore GH-\( 0 \). This result is consistent with Remark \ref{elliptic_is_GH-0}, where we established that elliptic Fourier multipliers are always GH-\( 0 \).
	\end{example}

	\begin{example}\label{Heat_GH-1}
		The heat operator \( p(D) = \partial_{x_1} - \Delta_{\mathbb{T}^n} \) on \( \mathbb{T}^{n+1} \) is not elliptic and has symbol
		\[
		p(\xi) = i\xi_1 + \sum_{j=2}^{n+1} \xi_j^2.
		\]
		
		A direct analysis of the symbol shows that for \( \xi \in \mathbb{Z}^{n+1} \),
		\[
		|p(\xi)|^2 = \xi_1^2 + \left( \sum_{j=2}^{n+1} \xi_j^2 \right)^2 \geq \|\xi\|^2 \geq \frac{1}{n+1} |\xi|^2.
		\]
		
		Taking square roots, we obtain
		\[
		|p(\xi)| \geq (n+1)^{-\frac{1}{2}} |\xi| = (n+1)^{-\frac{1}{2}} |\xi|^{2 - 1}, \ \xi \in \mathbb{Z}^{n+1}.
		\]
		
		On the other hand, for \( \xi \in \mathbb{Z}^{n+1} \) such that \( \xi_2 = \dots = \xi_{n+1} = 0 \), we have
		\[
		|p(\xi)| = |i\xi_1| = |\xi_1| = |\xi| \leq |\xi|^{2 - 1},
		\]
		which holds for infinitely many \( \xi \in \mathbb{Z}^{n+1} \setminus \{0\} \).
		
		By Theorem \ref{theoGHindex}, these estimates imply that the heat operator is GH-\( 1 \), and its global hypoellipticity index satisfies
		\[
		\operatorname{ind}_{GH}(p(D)) = 1.
		\]

	 In particular, for any $k>\frac{n+1}{2}-1$, any solutions $u$ for the heat equation $\partial_{x_1}u-\Delta_{\T^n}u=f$, where $f\in C^{k}(\T^{n+1})$, must satisfy $u\in C^{k'}(\T^{n+1})$, where $k'=k-\frac{n+1}{2}$ if $n$ is odd and $k'=k-\frac{n}{2}$ if $n$ is even.
	\end{example}

	\begin{example}\label{Wave_GH-infty}
		The wave operator \( p(D) = \partial_{x_1}^2 - \eta^2 \Delta_{\mathbb{T}^n} \), where \( \eta=a/b\in\mathbb{Q}_+ \), acts on \( \mathbb{T}^{n+1} \) and has the symbol
		\[
		p(\xi) = -\xi_1^2 + \frac{a^2}{b^2} \sum_{j=2}^{n+1} \xi_j^2.
		\]	
		This symbol vanishes on the set of frequencies satisfying
		\[
		b^2\xi_1^2 = {a^2} \|\xi'\|^2, \quad \xi' = (\xi_2, \dots, \xi_{n+1}) \in \mathbb{Z}^n.
		\]
		
		In particular, for \( \xi_3 = \dots = \xi_{n+1} = 0 \), this reduces to \( b\xi_1 = \pm a  \xi_2 \).		
		Since \( p(\xi) = 0 \) on this infinite set, condition {\it (i)} of Theorem \ref{theoGH} is violated. Consequently, \( p(D) \) fails to be GH-\( r \) for any \( r \geq 0 \), and its global hypoellipticity index is
		\[
		\operatorname{ind}_{GH}(p(D)) = \infty.
		\]
         In particular, solutions $u$ for the wave equation $\partial_{x_1}^2u-(a^2/b^2)\Delta_{\T^n}u=f$ can exhibit arbitrarily low regularity, regardless of the smoothness of $f$. 
	\end{example}

	Before stating the next result, we recall the concept of global hypoellipticity and its characterization for Fourier multipliers.
	
	\begin{definition}
		Let \( p \in S^{m}(\mathbb{Z}^n) \). The operator \( p(D) \) is said to be \textit{globally hypoelliptic} (denoted GH) if the following implication holds:
		\[
		u \in \mathscr{D}'(\mathbb{T}^n) \text{ and } p(D)u \in C^\infty(\mathbb{T}^n) \implies u \in C^\infty(\mathbb{T}^n).
		\]
	\end{definition}
	
	\begin{theorem}[Greenfield–Wallach]\label{theoGW}
		Let \( p \in S^{m}(\mathbb{Z}^n) \). The operator \( p(D) \) is globally hypoelliptic if and only if:
		\begin{enumerate}
			\item[{\it (i)}] The symbol \( p(\xi) \) vanishes for only finitely many \( \xi \in \mathbb{Z}^n \).
			
			\item[{\it (ii)}] There exist positive constants \( K, L > 0 \) such that
			\begin{equation}\label{ineqGW}
				|p(\xi)| \geq K |\xi|^{-L},
			\end{equation}
			for all \( \xi \in \mathbb{Z}^n \setminus \{0\} \) satisfying \( p(\xi) \neq 0 \).
		\end{enumerate}
	\end{theorem}
	
	The proof of this theorem follows the same reasoning as in the differential case discussed in \cite{GW1972_pams}.
		
	It is worth noting that the original Greenfield–Wallach condition for global hypoellipticity on the torus involves constants \( K, L, M > 0 \) such that
	\[
	|p(\xi)| \geq K |\xi|^{-L} \quad \text{for all } |\xi| > M.
	\]
	
	Our formulation is equivalent and aligns with the approach we have adopted in previous studies for vector fields on Lie groups \cite{KMR2020_bsm, KMR2021_jfa, Kiri_Kow_Wagn_ToroEsfera} and pseudodifferential operators on closed manifolds \cite{AGK2018_jam, AK2019_jst}.

	\begin{corollary}\label{coro_GH_equiv}
		Let \( p \in S^{[m]}(\mathbb{Z}^n) \). The following statements are equivalent:
		\begin{enumerate}
			\item[{\it (i)}] \( p(D) \) is globally hypoelliptic (GH).
			\item[{\it (ii)}] \( p(D) \) is GH-\( r \) for some \( r \geq 0 \).
            \item[{\it (iii)}] \( \operatorname{ind}_{GH}(p(D)) < \infty \).
		\end{enumerate}
	\end{corollary}
	
	\begin{proof}
		By definition, {\it (iii)} is equivalent to {\it (ii)}, so it suffices to prove that {\it (i)} is equivalent to {\it (ii)}.
		
		First, assume \( p(D) \) is GH-\( r \) for some \( r \geq 0 \). Then by Theorem \ref{theoGH}, the symbol $p(\xi)$ vanishes for only finitely many $\xi\in\Z^n$ and there exists $K>0$ such that inequality \ref{ineqGW} holds for $L=-m+r$. Therefore by Theorem \ref{theoGW} we have that $p(D)$ is globally hypoelliptic.
		
		Conversely, assume \( p(D) \) is globally hypoelliptic. By Theorem \ref{theoGW}, \( p(\xi) \) vanishes for only finitely many \( \xi \in \mathbb{Z}^n \), and there exist constants \( K, L > 0 \) such that
		\[
		|p(\xi)| \geq K |\xi|^{-L}, \quad \text{for all } \xi \neq 0 \text{ with } p(\xi) \neq 0.
		\]
		
		Notice that the inequality above still holds if we replace \( L \) with any larger constant. Thus, choosing \( L \) sufficiently large so that \( L \geq -m \), Theorem \ref{theoGH} implies that \( p(D) \) is GH-\( r \) for \( r = m + L \geq 0 \). Consequently, we conclude that \( \operatorname{ind}_{GH}(p(D)) < \infty \).  		
	\end{proof}
	
	Next, we demonstrate that verifying global hypoellipticity of order \( r \) for Fourier multipliers can be reduced to checking a single Sobolev order.
	
    \begin{proposition}\label{prop_single_sobolev_order}
    	Let \( p \in S^{[m]}(\mathbb{Z}^n) \). The operator \( p(D) \) is GH-\( r \) if and only if there exists a fixed \( k \in \mathbb{R} \) such that for every \( u \in \mathscr{D}'(\mathbb{T}^n) \), the condition \( p(D)u \in H^k(\mathbb{T}^n) \) implies \( u \in H^{k + m - r}(\mathbb{T}^n) \).
    \end{proposition}
    
    \begin{proof}
    	First, assume \( p(D) \) is GH-\( r \). By definition, there exists \( k \in \mathbb{R} \) satisfying the stated implication.
    	
    	Conversely, suppose there exists a fixed \( k \in \mathbb{R} \) such that for every \( u \in \mathscr{D}'(\mathbb{T}^n) \), the condition \( p(D)u \in H^k(\mathbb{T}^n) \) implies \( u \in H^{k + m - r}(\mathbb{T}^n) \). We now show that this property extends to all Sobolev orders.
    	
    	Consider the Bessel potential of order \( s \in \mathbb{R} \):
    	\[
    	\Lambda^s \coloneqq (I - \Delta)^{-s/2}: H^k(\mathbb{T}^n) \to H^{k+s}(\mathbb{T}^n),
    	\]
    	which is a Fourier multiplier with symbol \( (1 + \|\xi\|^2)^{-s/2} \). The operator \( \Lambda^s \) is a bijection between \( H^k(\mathbb{T}^n) \) and \( H^{k+s}(\mathbb{T}^n) \) for any \( k \in \mathbb{R} \).
    	
    	Let \( u \in \mathscr{D}'(\mathbb{T}^n) \) satisfy \( p(D)u = f \in H^{k'}(\mathbb{T}^n) \) for some arbitrary \( k' \in \mathbb{R} \). Write \( k = k' + s \) for some \( s \in \mathbb{R} \). Since Fourier multipliers commute, we have:
    	\[
    	p(D)\circ \Lambda^s u = \Lambda^s \circ p(D)u = \Lambda^s f \in H^k(\mathbb{T}^n).
    	\]
    	
    	By assumption, \( \Lambda^s u \in H^{k + m - r}(\mathbb{T}^n) \), which implies \[ u \in H^{k + m - r - s}(\mathbb{T}^n) = H^{k' + m - r}(\mathbb{T}^n). \]
    	
    	Thus, the property holds for all \( k' \in \mathbb{R} \), proving that \( p(D) \) is GH-\( r \).
    \end{proof}
    
    \begin{corollary}\label{coro_ind_gr_char}
    	Let \( p \in S^{[m]}(\mathbb{Z}^n) \). Then \( \operatorname{ind}_{GH}(p(D)) = r \) if and only if \( r \) is the infimum of all \( s \geq 0 \) such that there exists \( k \in \mathbb{R} \) satisfying the following implication for all \( u \in \mathscr{D}'(\mathbb{T}^n) \):
    	\[
    	p(D)u \in H^k(\mathbb{T}^n) \implies u \in H^{k + m - s}(\mathbb{T}^n).
    	\]
    \end{corollary}
	
	\section{Solvability with Loss of Derivatives}
	
	We now introduce a notion of solvability in lower regularity spaces, intrinsically linked to the global hypoellipticity with loss of derivatives analyzed in the previous section.
	
	\begin{definition}\label{def_GS}
		Let \( p \in S^{[m]}(\mathbb{Z}^n) \) and \( r \geq 0 \). The operator \( p(D) \) is \textit{globally solvable with a loss of \( r \) derivatives} (denoted GS-\( r \)) if, for every \( k \in \mathbb{R} \) and every \( f \in H^k(\mathbb{T}^n) \cap \left( \ker {}^t p(D) \cap C^\infty(\mathbb{T}^n) \right)^0 \), there exists \( u \in H^{k + m - r}(\mathbb{T}^n) \) such that \( p(D)u = f \).
		
		Additionally, the \textit{global solvability index} of \( p(D) \) is defined as
		\[
		\operatorname{ind}_{GS}(p(D)) \vcentcolon= \inf \left\{ r \geq 0 : p(D) \text{ is GS-}r \right\},
		\]
		with the convention that \( \operatorname{ind}_{GS}(p(D)) = \infty \) if \( p(D) \) is not GS-\( r \) for any \( r \geq 0 \).
	\end{definition}

	\begin{remark}
		As in the case of global hypoellipticity with loss of derivatives, the definition of the property of \( p(D) \) being GS-\( r \) may also depend on the choice of the operator's order. This justifies the need to specify the intrinsic order in the definition.
		
		A direct consequence of this definition and the Sobolev embedding theorem, we observe that if \( p(D) \) is GS-\( r \), then for any \( k \in \mathbb{N}_0 \) satisfying \( k >r+\frac{n}{2}-m \), every admissible function \( f \in C^k(\mathbb{T}^n) \) admits a solution \( u \in C^{k'}(\mathbb{T}^n) \) to the equation \( p(D)u = f \),  where $k'$ denotes the greatest integer {\bf less than} $k+m-r-\frac{n}{2}$.
	\end{remark}
	
	The next result provides a characterization of the annihilator of the kernel of the transpose operator in terms of Fourier coefficients. This characterization will play a key role in the subsequent results.
	
	\begin{lemma}\label{equivalence_ker0}
		Let \( p \in S^{m}(\mathbb{Z}^n) \) and \( f \in H^k(\mathbb{T}^n) \). Then  
		\[
		f \in (\ker \pt \cap C^\infty(\mathbb{T}^n))^0  \Longleftrightarrow  \widehat{f}(\xi) = 0,  \forall \xi \in \mathbb{Z}^n \text{ such that } p(\xi) = 0.
		\]
	\end{lemma}

	\begin{proof}
		First, recall that the transpose operator \( \pt \) is also a Fourier multiplier of the same order as \( p(D) \), with symbol given by  
		\[
		\prescript{t}{}{p}(\xi) = p(-\xi), \quad \xi \in \mathbb{Z}^n,
		\]
		as shown in \cite[Theorem 2.3]{FerraPetronilho2021}.
		
		To prove the ``only if" implication, suppose that \( p(\xi_0) = 0 \) for some \( \xi_0 \in \mathbb{Z}^n \). Consider the smooth function \( \phi \) defined by its Fourier coefficients as follows:
		\[
		\widehat{\phi}(\xi) =
		\begin{cases}
			1, & \text{if } \xi = -\xi_0, \\
			0, & \text{otherwise}.
		\end{cases}
		\]
		
		Then, by applying \( \pt \) to \( \phi \), we obtain  
		\[
		\pt \phi(x) = \sum_{\xi \in \mathbb{Z}^n} e^{i x \cdot \xi} p(-\xi) \widehat{\phi}(\xi) = e^{-i x \cdot \xi_0} p(\xi_0) \widehat{\phi}(-\xi_0) = 0.
		\]  
		This shows that \( \phi \in \ker \pt \cap C^\infty(\mathbb{T}^n) \).  
		
		Now, if \( f \in (\ker \pt \cap C^\infty(\mathbb{T}^n))^0 \), then by definition,
		\[
		\langle f, \phi \rangle = 0.
		\]
		
		Using the Fourier representation, we obtain  
		\[
		(2\pi)^{n} \sum_{\xi \in \mathbb{Z}^n} \widehat{f}(\xi) \widehat{\phi}(-\xi) = (2\pi)^{n} \widehat{f}(\xi_0) = 0.
		\]
		
		Thus, \( \widehat{f}(\xi_0) = 0 \). Since \( \xi_0 \) was arbitrary, we conclude that \( \widehat{f}(\xi) = 0 \) for all \( \xi \in \mathbb{Z}^n \) such that \( p(\xi) = 0 \).
		
		\medskip
		Conversely, suppose that \( f \) satisfies \( \widehat{f}(\xi) = 0 \) whenever \( p(\xi) = 0 \). Consider any \( \phi \in \ker \pt \cap C^\infty(\mathbb{T}^n) \). Notice that it satisfies 
		\[
		 0=\widehat{\pt \phi}(-\xi) = p(\xi) \widehat{\phi}(-\xi), \ \xi \in \mathbb{Z}^n.
		\]  
		
		Therefore \( \widehat{\phi}(-\xi) = 0 \) whenever \( p(\xi) \neq 0 \). We conclude that  
		\begin{align*}
			\langle f, \phi \rangle &= (2\pi)^n \sum_{\xi \in \mathbb{Z}^n} \widehat{f}(\xi) \widehat{\phi}(-\xi) = 0,
		\end{align*}
		since every term in the sum above is equal to zero.\\ \noindent Therefore, \( f \in (\ker \pt \cap\, C^\infty(\mathbb{T}^n))^0 \), which completes the proof.
	\end{proof}
	
	\begin{theorem}\label{theoGS}
		Let \( r \geq 0 \) and \( p \in S^{[m]}(\mathbb{Z}^n) \). The operator \( p(D) \) is GS-$r$ if and only if there exists a constant \( K > 0 \) such that
		\begin{equation}\label{ineq-GS}
			|p(\xi)|\geq K|\xi|^{m-r},
		\end{equation}
		for every \( \xi \in \mathbb{Z}^n \setminus \{0\} \) such that \( p(\xi) \neq 0 \).
	\end{theorem}
	
	\begin{proof}
		Suppose that inequality \eqref{ineq-GS} does not hold. Then, there exists a sequence of distinct elements \( (\xi_j)_{j \in \mathbb{N}} \) in \( \mathbb{Z}^n \setminus \{0\} \) such that
		\begin{equation}\label{ineq-GS-not}
			0 < |p(\xi_j)|\leq  \frac{1}{j}|\xi_j|^{m-r}, \quad  \forall j \in \mathbb{N}.
		\end{equation}	
		Define a distribution \( f \in \mathscr{D}'(\mathbb{T}^n) \) by its Fourier coefficients:
\[
		\widehat{f}(\xi) = 
		\begin{cases}
			p(\xi_j), & \text{if } \xi = \xi_j \text{ for some } j \in \mathbb{N}, \\
			0, & \text{otherwise}.
		\end{cases}
		\]
            Note that
		\begin{align*}
			\|f\|_{H^{-m + r}(\mathbb{T}^n)}^2 
			&= \sum_{\xi \in \mathbb{Z}^n} (1 + \|\xi\|^2)^{-m + r} |\widehat{f}(\xi)|^2 \\
			&= \sum_{j=1}^\infty (1 + \|\xi_j\|^2)^{-m + r} |p(\xi_j)|^2 \\
			&\leq \sum_{j=1}^\infty (1 + \|\xi_j\|^2)^{-m + r} \cdot \frac{1}{j^2} |\xi_j|^{2(m - r)} \\
			&\leq (1 + n)^{|m - r|} \sum_{j=1}^\infty \frac{1}{j^2} < \infty.
		\end{align*}
		where, in the last inequality, we used the bound  
	\begin{equation}\label{good-bound}
			|\xi|^{-2\tau} \leq (1+n)^{|\tau|} (1+\|\xi\|^2)^{-\tau}, \quad \forall \xi \in \mathbb{Z}^n \setminus \{0\}, \quad \forall \tau \in \mathbb{R}.
	\end{equation}
		Thus, \( f \in H^{-m + r}(\mathbb{T}^n) \). Moreover, since \( \widehat{f}(\xi) = 0 \) whenever \( p(\xi) = 0 \), Lemma \ref{equivalence_ker0} ensures that \( f \in (\ker {}^t p(D) \cap C^\infty(\mathbb{T}^n))^0 \).

        However, if $p(D)u=f$, for some $u\in\mathscr{D}'(\T^n)$, the its Fourier coefficients must satisfy
		\[
		\widehat{u}(\xi) = 
		\begin{cases}
			1, & \text{if } \xi = \xi_j \text{ for some } j \in \mathbb{N}, \\
			0, & \text{otherwise}.
		\end{cases}
		\]
		But then, since
		\[
		\|u\|_{H^0(\mathbb{T}^n)}^2 = \sum_{j=1}^\infty |\widehat{u}(\xi_j)|^2 = \sum_{j=1}^\infty 1 = \infty,
		\]
		we conclude that \( u \notin H^0(\mathbb{T}^n) \).
			
		This confirms that for the given \( f \in H^{-m + r}(\mathbb{T}^n) \cap \left( \ker {}^t p(D) \cap C^\infty \right)^0 \), there exists no solution \( u \in H^{0}(\mathbb{T}^n) \) to the equation \( p(D)u = f \). Hence, the operator \( p(D) \) fails to be GS-\( r \).

		\medskip		
		Conversely, assume that inequality \eqref{ineq-GS} holds and let us show that \( p(D) \) is GS-\( r \). 
		
		Let \( f \in H^{k}(\mathbb{T}^n) \cap \left( \ker {}^t p(D) \cap C^\infty(\mathbb{T}^n) \right)^0 \) and define \( u \in \mathscr{D}'(\mathbb{T}^n) \) by its Fourier coefficients:
		\[
		\widehat{u}(\xi) = 
		\begin{cases}
			\dfrac{\widehat{f}(\xi)}{p(\xi)}, & \text{if } p(\xi) \neq 0, \\
			0, & \text{if } p(\xi) = 0.
		\end{cases}
		\]
		
		By Lemma \ref{equivalence_ker0} and by comparing Fourier coefficients, we have that \( p(D)u = f \). Moreover, 
		\begin{align*}
			\|u\|_{H^{k + m - r}(\mathbb{T}^n)}^2 
			&= \sum_{\substack{\xi \in \mathbb{Z}^n \\ p(\xi) \neq 0}} (1 + \|\xi\|^2)^{k + m - r} \left| \frac{\widehat{f}(\xi)}{p(\xi)} \right|^2 \\
			&\leq \sum_{\substack{\xi \in \mathbb{Z}^n \\ p(\xi) \neq 0}} (1 + \|\xi\|^2)^{k + m - r} \cdot \frac{|\widehat{f}(\xi)|^2}{K^2 |\xi|^{2(m - r)}} \quad \text{(by \eqref{ineq-GS})} \\
			& \leq \frac{(1 + n)^{|m - r|}}{K^2} \sum_{\xi \in \mathbb{Z}^n} (1 + \|\xi\|^2)^{k} |\widehat{f}(\xi)|^2  \quad \text{(by \eqref{good-bound})} \\
			&= \frac{(1 + n)^{|m - r|}}{K^2} \|f\|_{H^k(\mathbb{T}^n)}^2 < \infty.
		\end{align*}
		
		Thus, \( u \in H^{k + m - r}(\mathbb{T}^n) \), proving that \( p(D) \) is GS-\( r \). 
	\end{proof}
      
      \
      
   \begin{corollary}\label{coroGHimpliesGS}
   	Let \(r\geq0\) and \( p \in S^{[m]}(\mathbb{Z}^n) \). The following properties hold:
   	\begin{enumerate}
   		\item[{\it (i)}] If \( p(D) \) is GH-\( r \), then it is GS-\( r \).
   		\item[{\it (ii)}] \smallskip If \( p(\xi) \) vanishes for only finitely many \( \xi \in \mathbb{Z}^n \), then 
   		\[ p(D)  \text{ is GH-}r \Longleftrightarrow   p(D)  \text{ is GS-}r. \]
   		\item[{\it (iii)}] If \( \operatorname{ind}_{GH}(p(D)) < \infty \), then 
   		\[ \operatorname{ind}_{GS}(p(D)) = \operatorname{ind}_{GH}(p(D)). \]
   	\end{enumerate}
   \end{corollary}
	
	\begin{proof}
		The first two statements follow directly from Theorems \ref{theoGH} and \ref{theoGS}.
		For {\it (iii)}, note that if \( \operatorname{ind}_{GH}(p(D)) < \infty \), then by Corollary \ref{coro_GH_equiv}, \( p(D) \) is GH-\( r \) for some finite \( r \). Statement {\it (ii)} then also holds  (via Theorem \ref{theoGH}), and we conclude that \( \operatorname{ind}_{GS}(p(D)) = \operatorname{ind}_{GH}(p(D)) \).
	\end{proof}

	\begin{example}
		As a consequence of Corollary \ref{coroGHimpliesGS}, every elliptic operator is GS-\( 0 \), and its global solvability index satisfies \( \operatorname{ind}_{GS}(p(D)) = 0 \).
		
		In Examples \ref{Laplacian_GH-0} and \ref{Heat_GH-1}, we established that the Laplacian is GH-\( 0 \), while the heat operator is GH-\( 1 \). Specifically:
		\[
		\operatorname{ind}_{GH}(\Delta_{\mathbb{T}^n}) = 0, \quad \operatorname{ind}_{GH}(\partial_{x_1} - \Delta_{\mathbb{T}^n}) = 1.
		\]
		
		By Corollary \ref{coroGHimpliesGS}, these operators also satisfy the corresponding global solvability properties. Thus:
		\[
		\operatorname{ind}_{GS}(\Delta_{\mathbb{T}^n}) = 0, \quad \operatorname{ind}_{GS}(\partial_{x_1} - \Delta_{\mathbb{T}^n}) = 1.
		\]

        In particular, for any $k>\frac{n+1}{2}-1$ and $f\in C^k(\T^{n+1})$, the heat equation $\partial_{x_1}u-\Delta_{\T^n}u=f$, admits solution $u\in C^{k-\lceil\frac{n}{2}\rceil}(\T^{n+1})$, if $\widehat{f}(0)=0$.  
        
		The wave operator \( p(D) = \partial_{x_1}^2 - \eta^2 \Delta_{\mathbb{T}^n} \), however, behaves differently. As shown in Example \ref{Wave_GH-infty}, \(\operatorname{ind}_{GH}(p(D)) = \infty\) and so Corollary \ref{coroGHimpliesGS} is not enough to characterize its global solvability properties.
		
		We will revisit the wave operator in the final section to provide further insights into its behavior.
	\end{example}

	\begin{corollary}\label{coroGS1}
		Let \( p \in S^{[m]}(\mathbb{Z}^n) \) and \( r' > 0 \). If there exists a constant \( K' > 0 \) such that
		\begin{equation} \label{ineq-not-GS-coro}
			0 < |p(\xi)| \leq K' |\xi|^{m - r'}
		\end{equation}
		for infinitely many \( \xi \in \mathbb{Z}^n \setminus \{0\} \), then \( p(D) \) is not GS-\( r \) for any \( 0\leq r < r' \). Consequently,
		\[
			\operatorname{ind}_{GS}(p(D)) \geq r'.
		\]
	\end{corollary}

	\begin{proof}
		Assume \( p(\xi) \) satisfies \eqref{ineq-not-GS-coro} for infinitely many \( \xi \in \mathbb{Z}^n \setminus \{0\} \). Let \( 0\leq r < r' \), and write \( r' = r + \varepsilon \) for some \( \varepsilon > 0 \). Then, we can rewrite the inequality as:
		\[
		0 < |p(\xi)| \leq K' |\xi|^{m - r} |\xi|^{-\varepsilon},
		\]
		which holds for infinitely many \( \xi \).
		
		Among such infinitely many $\xi$, choose a sequence \( (\xi_j)_{j \in \mathbb{N}} \subset \mathbb{Z}^n \setminus \{0\} \) such that \( |\xi_j| \geq (K' j)^{1/\varepsilon} \). For these \( \xi_j \), we have:
		\[
		0<|p(\xi_j)| \leq K' |\xi_j|^{m - r} |\xi_j|^{-\varepsilon} \leq \frac{1}{j} |\xi_j|^{m - r}.
		\]
		
		As shown in the proof of Theorem \ref{theoGS}, this estimate implies that \( p(D) \) fails to satisfy the GS-\( r \) property. Therefore, \( p(D) \) is not GS-\( r \) for any \( 0\leq r < r' \), and it follows that:
		\[
		\operatorname{ind}_{GS}(p(D)) \geq r'.
		\]
		This completes the proof.
	\end{proof}

		\begin{theorem}\label{theoGSindex}
		Let \( r \geq 0 \) and \( p \in S^{[m]}(\mathbb{Z}^n) \). Then the global solvability index of \( p(D) \) satisfies \( \operatorname{ind}_{GS}(p(D)) = r \) if and only if the following conditions hold:
		\begin{enumerate}
			\item[{\it (i)}] For every \( \varepsilon > 0 \), there exists a constant \( K_\varepsilon > 0 \) such that
			\begin{equation}\label{ineqGS_lower}
				|p(\xi)| \geq K_\varepsilon |\xi|^{m - (r + \varepsilon)},
			\end{equation}
			for all \( \xi \in \mathbb{Z}^n \setminus \{0\} \) satisfying \( p(\xi) \neq 0 \).
			
			\item[{\it (ii)}] If \( r > 0 \), then for every \( \varepsilon > 0 \), there exists a constant \( K_{\varepsilon}' > 0 \) such that
			\begin{equation}\label{ineqGS_upper}
				0 < |p(\xi)| \leq K_{\varepsilon}' |\xi|^{m - (r - \varepsilon)},
			\end{equation}
			for infinitely many \( \xi \in \mathbb{Z}^n \setminus \{0\} \).
		\end{enumerate}
	\end{theorem}

	\begin{proof}
		We begin by proving the sufficiency. Suppose that conditions {\it (i)} and {\it (ii)} hold. By Theorem \ref{theoGS}, condition {\it (i)} guarantees that \( p(D) \) is GS-\( (r + \varepsilon) \) for every \( \varepsilon > 0 \), implying that  
		\[
		\operatorname{ind}_{GS}(p(D)) \leq r.
		\]
		
		If \( r = 0 \), then we immediately conclude that \( \operatorname{ind}_{GS}(p(D)) = 0 \). Now, if \( r > 0 \), condition {\it (ii)} together with Corollary \ref{coroGS1} ensures that \( p(D) \) is not GS-\( s \) for any \( s < r \), leading to  
		\[
		\operatorname{ind}_{GS}(p(D)) \geq r.
		\]
		Thus, we conclude that \( \operatorname{ind}_{GS}(p(D)) = r \).
		
		\medskip
		Next, we prove the necessity. Suppose that \( \operatorname{ind}_{GS}(p(D)) = r \). If condition {\it (i)} fails, then there exists some \( \varepsilon > 0 \) such that 
		\begin{equation*}
			0<|p(\xi)| \leq \frac{1}{j} |\xi|^{m - (r + \varepsilon)},
		\end{equation*}
        for infinitely many \( \xi \in \mathbb{Z}^n \setminus \{0\} \). By Corollary \ref{coroGS1}, this would imply that \( \operatorname{ind}_{GS}(p(D)) \geq r + \varepsilon \), contradicting the assumption that \( \operatorname{ind}_{GS}(p(D)) = r \). Therefore, condition {\it (i)} must hold. If $r=0$, this concludes the proof of necessity.
		
		Now, assume that \( r > 0 \) and condition {\it (ii)} fails. Then there exists \( \varepsilon > 0 \) such that  
		\[
		0<|p(\xi)| \leq |\xi|^{m - (r - \varepsilon)}
		\]
	 for only finitely many \( \xi \). In this case, we can find \( K > 0 \) sufficiently large such that  
		\[
		|p(\xi)| \geq K |\xi|^{m - (r - \varepsilon)},
		\]
		for all \( \xi \in \mathbb{Z}^n \setminus \{0\} \) where \( p(\xi) \neq 0 \). By Theorem \ref{theoGS}, this implies that \( p(D) \) is GS-\( (r - \varepsilon) \), contradicting \( \operatorname{ind}_{GS}(p(D)) = r \). Thus, condition {\it (ii)} must also hold, completing the proof.
	\end{proof}

	Next, we show that verifying global solvability with a loss of \( r \) derivatives reduces to checking the condition from Theorem \ref{theoGS} for a single Sobolev order. This simplification parallels the approach used for the GH-\( r \) property.
	
	\begin{proposition}\label{prop_single_sobolev_order_GS}
		Let \( p \in S^{[m]}(\mathbb{Z}^n) \) and \( r \geq 0 \). The operator \( p(D) \) is GS-\( r \) if and only if there exists  \( k \in \mathbb{R} \) such that for every function \( f \in H^k(\mathbb{T}^n) \cap \left( \ker {}^t p(D) \cap C^\infty(\mathbb{T}^n) \right)^0 \), there exists \( u \in H^{k + m - r}(\mathbb{T}^n) \) satisfying \( p(D)u = f \).
	\end{proposition}
	
	\begin{proof}
		The necessity follows directly from the definition of GS-\( r \). 
		For sufficiency, assume there exists \( k \in \mathbb{R} \) such that for every \( f \in H^k(\mathbb{T}^n) \cap \left( \ker {}^t p(D) \cap C^\infty(\mathbb{T}^n) \right)^0 \), there exists \( u \in H^{k + m - r}(\mathbb{T}^n) \) satisfying \( p(D)u = f \).
		
		Let \( f \in H^{k'}(\mathbb{T}^n) \cap \left( \ker {}^t p(D) \cap C^\infty(\mathbb{T}^n) \right)^0 \) for some arbitrary \( k' \in \mathbb{R} \). Write \( k' = k + s \) with \( s \in \mathbb{R} \), and let \( \Lambda^s: H^k(\mathbb{T}^n) \to H^{k+s}(\mathbb{T}^n) \) denote the Bessel potential of order \( s \), as in the proof of Proposition \ref{prop_single_sobolev_order}. 
		
		Define \( g \coloneqq \Lambda^{-s} f \), which belongs to \( H^k(\mathbb{T}^n) \).	Since \( \widehat{g}(\xi) = (1 + |\xi|^2)^{-s/2} \widehat{f}(\xi) \), Lemma \ref{equivalence_ker0} implies \( g \in \left( \ker {}^t p(D) \cap C^\infty(\mathbb{T}^n) \right)^0 \). By hypothesis, there exists \( u \in H^{k + m - r}(\mathbb{T}^n) \) such that \( p(D)u = g \).
		
		Define \( v \coloneqq \Lambda^s u \). Since Fourier multipliers commute with \( \Lambda^s \), we have:
		\[
		p(D)v = p(D) \Lambda^s u = \Lambda^s p(D)u = \Lambda^s g = f.
		\]

		Moreover, \( v \in H^{k + m - r + s}(\mathbb{T}^n) = H^{k' + m - r}(\mathbb{T}^n) \). As \( k' \) was arbitrary, this proves that \( p(D) \) is GS-\( r \).
	\end{proof}
	
	As discussed in the introduction, a classical notion of global solvability requires an operator \( p(D): C^\infty(M) \to C^\infty(M) \) to have closed range. The following theorem establishes that our definition of GS-\( r \) solvability is consistent with this classical perspective.
	
	\begin{theorem}\label{teoGS-closed}
		Let \( p \in S^{[m]}(\mathbb{Z}^n) \) and \( r \geq 0 \). The operator \( p(D) \) is GS-\( r \) if and only if for every \( k \in \mathbb{R} \), the mapping
		\[
		p(D): H^{k + m - r}(\mathbb{T}^n) \to H^k(\mathbb{T}^n)
		\]
		has closed range.
	\end{theorem}
	
		\begin{proof}
		Suppose that $p(D)$ is GS-$r$. Then by Theorem \ref{theoGS} we have that there exists $K>0$ such that
		\begin{align}\label{ineq-GS-strong}
			|p(\xi)|\geq K|\xi|^{m-r},
		\end{align}
		for every $\xi\in\Z^n\backslash\{0\}$ such that $p(\xi)\neq 0$.
		
		First notice that for $m'>m$, $p(D)$ if of order $m'$. Therefore $p(D)$ maps $H^{k + m - r}(\mathbb{T}^n)\subset H^{k+m'}(\mathbb{T}^n)$ to $H^k(\mathbb{T}^n)$, and the mapping \(p(D): H^{k + m - r}(\mathbb{T}^n) \to H^k(\mathbb{T}^n)\) is well defined. Now Let $ u \in H^{k + m - r}(\mathbb{T}^n) $. Then 
		\begin{align*}
			\|p(D)u\|_{H^k(\mathbb{T}^n)}^2 &= \sum_{\xi \in \mathbb{Z}^n} (1 + \|\xi\|^2)^k \left| \widehat{p(D)u}(\xi) \right|^2 \\
			&= \sum_{\xi \in \mathbb{Z}^n} (1 + \|\xi\|^2)^k |p(\xi)|^2 |\widehat{u}(\xi)|^2 \\
			&\geq K(1+n)^{|m-r|} \sum_{\xi \in \mathbb{Z}^n} (1 + \|\xi\|^2)^{k + m - r} |\widehat{u}(\xi)|^2 \\
			&= K(1+n)^{|m-r|} \|u\|_{H^{k + m - r}(\mathbb{T}^n)}^2,
		\end{align*}
		
		Since both $ H^k(\mathbb{T}^n) $ and $ H^{k + m - r}(\mathbb{T}^n) $ are Banach spaces, it follows that $ p(D):H^{k + m - r}(\mathbb{T}^n)\to H^k(\mathbb{T}^n)$ has closed range.

	Conversely, suppose that $ p(D) $ is not GS-$ r $. Then, by Theorem \ref{theoGS}, inequality \eqref{ineq-GS-strong} does not hold for any $ K > 0 $. Consequently, there exists a sequence of distinct elements $ \xi_j \in \mathbb{Z}^n \setminus \{0\} $ such that
	\begin{equation}\label{ineq_contradiction_closed}
		0 < |p(\xi_j)| \leq \frac{1}{j} |\xi_j|^{m - r}, \quad  j \in \mathbb{N}.
	\end{equation}
	
	Define a sequence of distributions $ (u_\ell)_{\ell \in \mathbb{N}} \in \mathscr{D}'(\mathbb{T}^n) $ whose Fourier coefficients are given by:
	\[
	\widehat{u_\ell}(\xi) =
	\begin{cases}
		(1 + \|\xi_j\|^2)^{(-k - m + r)/2}, & \text{if } \xi = \xi_j, \, j \leq \ell, \\
		0, & \text{otherwise.}
	\end{cases}
	\]
	
	For each $ \ell \in \mathbb{N} $, the Fourier coefficients of $ p(D)u_\ell $ are:
	\[
	\widehat{p(D)u_\ell}(\xi) =
	\begin{cases}
		p(\xi_j)(1 + \|\xi_j\|^2)^{(-k - m + r)/2}, & \text{if } \xi = \xi_j, \, j \leq \ell, \\
		0, & \text{otherwise.}
	\end{cases}
	\]
	
	Since both $ p(D)u_\ell $ and $ u_\ell $ have only finitely many nonzero Fourier coefficients, it follows that $ p(D)u_\ell, u_\ell \in \bigcap_{s \in \mathbb{R}} H^s(\mathbb{T}^n) = C^\infty(\mathbb{T}^n) $.
	
	Now consider the sequence of Fourier coefficients defined by:
	\[
	\widehat{f}(\xi) =
	\begin{cases}
		p(\xi_j)(1 + \|\xi_j\|^2)^{(-k - m + r)/2}, & \text{if } \xi = \xi_j, \, j \in \mathbb{N}, \\
		0, & \text{otherwise.}
	\end{cases}
	\]
	
	We claim that these coefficients define a distribution $ f \in H^k(\mathbb{T}^n) $. Indeed, we compute:
	\begin{align*}
		\|f\|_{H^k(\mathbb{T}^n)}^2 &= \sum_{\xi \in \mathbb{Z}^n} (1 + \|\xi\|^2)^k |\widehat{f}(\xi)|^2 \\
		&= \sum_{j \in \mathbb{N}} (1 + \|\xi_j\|^2)^k |p(\xi_j)|^2 (1 + \|\xi_j\|^2)^{-k - m + r} \\
		&\leq (1 + n)^{|m - r|} \sum_{j \in \mathbb{N}} \frac{1}{j^2} < \infty.
	\end{align*}
	
	Thus, $ f \in H^k(\mathbb{T}^n) $.
	
	\smallskip
	Next, observe that $ p(D)u_\ell \to f $ in $ H^k(\mathbb{T}^n) $ as $ \ell \to \infty $. Indeed, we have:
	\[
	\|p(D)u_\ell - f\|_{H^k(\mathbb{T}^n)}^2 = \sum_{j \geq \ell} (1 + \|\xi_j\|^2)^k |p(\xi_j)|^2.
	\]
	
	Since $ \sum_{j \geq 1} (1 + \|\xi_j\|^2)^k |p(\xi_j)|^2 $ converges, it follows from the Cauchy criterion for series that:
	\[
	\sum_{j \geq \ell} (1 + \|\xi_j\|^2)^k |p(\xi_j)|^2 \to 0 \quad \text{as } \ell \to \infty.
	\]
	
	Hence, $ \|p(D)u_\ell - f\|_{H^k(\mathbb{T}^n)}^2 \to 0 $, which implies $ p(D)u_\ell \to f $ in $ H^k(\mathbb{T}^n) $.
	
	Finally, note that $ f $ is not in the range of $ p(D): H^{k + m - r}(\mathbb{T}^n) \to H^k(\mathbb{T}^n) $. Indeed, if $ p(D)u = f $, then comparing Fourier coefficients yields:
	\[
	\widehat{u}(\xi) =
	\begin{cases}
		(1 + \|\xi_j\|^2)^{(-k - m + r)/2}, & \text{if } \xi = \xi_j, \, j \in \mathbb{N}, \\
		0, & \text{otherwise.}
	\end{cases}
	\]
	
	However, this implies that $ u \notin H^{k + m - r}(\mathbb{T}^n) $, since:
	\[
	\|u\|_{H^{k + m - r}(\mathbb{T}^n)}^2 = \sum_{j=1}^\infty (1 + \|\xi_j\|^2)^{(k + m - r)} (1 + \|\xi_j\|^2)^{(-k - m + r)} = \sum_{j=1}^\infty 1 = \infty.
	\]
	
	Therefore $ p(D): H^{k + m - r}(\mathbb{T}^n) \to H^k(\mathbb{T}^n) $ does not have closed range.
  \end{proof}
  
	\begin{corollary}\label{coro-GS2}
		Let \( r' > 0 \) and \( p \in S^{[m]}(\mathbb{Z}^n) \). 
		If there exists a constant \( K' > 0 \) such that
		\begin{equation}\label{ineq-not-GS-coro-2}
			0 < |p(\xi)| \leq K' |\xi|^{m - r'},
		\end{equation}
		for infinitely many \( \xi \in \mathbb{Z}^n \setminus \{0\} \), then the operator 
		\[ p(D): H^{k + m - r}(\mathbb{T}^n) \to H^k(\mathbb{T}^n) \]
		does not have closed range for any \( 0 \leq r < r' \) and \( k \in \mathbb{R} \).
		
		Consequently, if \( \operatorname{ind}_{GS}(p(D)) > r \), then \( p(D): H^{k + m - r}(\mathbb{T}^n) \to H^k(\mathbb{T}^n) \) does not have closed range for any \( k \in \mathbb{R} \).
	\end{corollary}
	
	\begin{proof}
		Suppose \( p \) satisfies inequality \eqref{ineq-not-GS-coro-2}. By Corollary \ref{coroGS1}, it follows that \( p(D) \) is not GS-\( r \) for any \( 0 \leq r < r' \). Consequently, by Theorem \ref{teoGS-closed}, the operator \( p(D): H^{k_0 + m - r}(\mathbb{T}^n) \to H^{k_0}(\mathbb{T}^n) \) does not have closed range for some \( k_0 \in \mathbb{R} \) and for any \( 0 \leq r < r' \).
		
		Now, let \( k' \in \mathbb{R} \) be arbitrary, and write \( k' = k_0 + s \) for some \( s \in \mathbb{R} \). Let \( \Lambda^s: H^{k}(\mathbb{T}^n) \to H^{k+s}(\mathbb{T}^n) \) denote the Bessel potential of order \( s \), as in the proof of Proposition \ref{prop_single_sobolev_order}. Since \( \Lambda^s \) is a homeomorphism, we have:
		\[
		\Lambda^s \circ p(D)(H^{k_0 + m - r}(\mathbb{T}^n)) \subset H^{k_0 + s}(\mathbb{T}^n) = H^{k'}(\mathbb{T}^n).
		\]
		
		Moreover, since Fourier multipliers commute with each other, it follows that:
		\begin{align*}
			\Lambda^s \circ p(D)(H^{k_0 + m - r}(\mathbb{T}^n)) & = p(D) \circ \Lambda^s(H^{k_0 + m -r} (\mathbb{T}^n)) \\
			&= p(D)(H^{k' + m - r}(\mathbb{T}^n)).			
		\end{align*}

		Thus, the range of \( p(D): H^{k' + m - r}(\mathbb{T}^n) \to H^{k'}(\mathbb{T}^n) \) is not closed, as it is homeomorphic to the range of \( p(D): H^{k_0 + m - r}(\mathbb{T}^n) \to H^{k_0}(\mathbb{T}^n) \), which is not closed. This completes the proof.
	\end{proof}

\section{Application: Vector Fields on the Two-Torus}
	
	In the early 1970s, N. Wallach and S. Greenfield published a series of seminal papers that established a connection between the global hypoellipticity of vector fields on the torus and Diophantine phenomena (see \cite{GW1972_pams, GW1973_tams, GW1973_top}). Their work laid the foundation for understanding how the regularity of solutions to partial differential equations on the torus is influenced by number-theoretic properties of the coefficients of the operators. 
	
	In this section, we revisit their results for vector fields on the two-torus \( \mathbb{T}^2 \), uncovering deeper relationships between the existence and regularity of solutions and fundamental concepts in number theory. These include the algebraic degree of irrational numbers and measures of irrationality, which play a central role in determining the hypoellipticity and solvability properties of such operators.
	
	To begin our analysis, consider the vector field defined on \( \mathbb{T}^2 \):
	\[
	p(D) = \partial_{x_1} - \alpha \partial_{x_2},
	\]
	where \( \alpha \in \mathbb{C} \) and \( (x_1, x_2) \in \mathbb{T}^2 \). The symbol of this operator is given by:
	\[
	p(\xi) = i(\xi_1 - \alpha \xi_2), \quad \xi = (\xi_1, \xi_2) \in \mathbb{Z}^2.
	\]
	
	When \( \operatorname{Im}(\alpha) \neq 0 \), we observe that \( p(\xi) \neq 0 \) for all \( \xi \in \mathbb{Z}^2 \setminus \{0\} \). This implies that the vector field is elliptic, and consequently, by Remark \ref{elliptic_is_GH-0}, it is globally hypoelliptic with no loss of derivatives (GH-\( 0 \)) and has global hypoellipticity index \( 0 \).
	
	The case \( \alpha \in \mathbb{R} \), i.e., \( \operatorname{Im}(\alpha) = 0 \), requires a more detailed investigation. In this setting, we have:
	\[
	|p(\xi)| = |\xi_1 - \alpha \xi_2| = |\xi_2| \left| \frac{\xi_1}{\xi_2} - \alpha \right|,
	\]
	for every \( \xi \in \mathbb{Z}^2 \) such that \( \xi_2 \neq 0 \).
	
	If  \( \alpha \) is rational, there exist infinitely many pairs \( (\xi_1, \xi_2) \in \mathbb{Z}^2 \) such that \( p(\xi) = 0 \). Consequently, the vector field \( p(D) \) fails to be GH-\( r \) for any \( r \geq 0 \).
	
	Now consider the case where \( \alpha \) is irrational. Within the class of irrational numbers, two important subclasses arise: algebraic numbers and transcendental numbers. 
	
	To provide a structured and insightful presentation, we begin by analyzing the case of algebraic numbers.
	
	\begin{definition}
		A complex number \( \lambda \) is said to be algebraic if it is a root of a nonzero polynomial with integer coefficients. The degree \( d \) of this algebraic number is the smallest possible degree of a polynomial that vanishes at \( \lambda \).
	\end{definition}

	Let \( \alpha \) be an algebraic irrational number of degree \( d \geq 2 \). Since \( \alpha \) is irrational, we have \( p(\xi) = 0 \) only when \( \xi = 0 \). Therefore, condition {\it (i)} of Theorem \ref{theoGH} is satisfied for any irrational \( \alpha \), and we only need to focus on condition {\it (ii)}, which concerns the asymptotic behavior of the symbol \( p(\xi) \).
	
	For \( \xi \in \mathbb{Z}^2 \) with \( \xi_2 \neq 0 \), the symbol can be expressed as:
	\begin{equation}\label{symbol-vf} 
		|p(\xi)| = |\xi_2| \left| \alpha - \frac{\xi_1}{\xi_2} \right|.
	\end{equation}
	
	By Liouville's theorem, there exists a constant \( K > 0 \), depending only on \( \alpha \), such that
	\[
	\left| \alpha - \frac{\xi_1}{\xi_2} \right| \geq \frac{K}{|\xi_2|^d},
	\]
	for all \( \xi_1, \xi_2 \in \mathbb{Z} \) with \( \xi_2 \neq 0 \). 
	
	When \( \xi_2 = 0 \), we have \( |p(\xi)| = |\xi_1| = |\xi| \), which satisfies \( |\xi| \geq |\xi|^{1-d} \) whenever \( \xi \neq 0 \). Combining these cases, we conclude that for all \( \xi \in \mathbb{Z}^2 \setminus \{0\} \),
	\[
	|p(\xi)| \geq \min\{K, 1\} |\xi|^{1-d}.
	\]
	
	Thus, by Theorem \ref{theoGH}, the vector field \( \partial_{x_1} - \alpha \partial_{x_2} \) is globally hypoelliptic with loss of \( d \) derivatives (GH-\( d \)) whenever \( \alpha \) is an algebraic irrational number of degree \( d \geq 2 \).
	
	The results established above lead to can be summarized as follows:

		Let \( p(D) = \partial_{x_1} - \alpha \partial_{x_2} \) be a vector field defined on the two-dimensional torus, where \( \alpha\in \mathbb{C} \). The following statements hold:
		\begin{enumerate}
			\item[{\it (i)}] If  \( \operatorname{Im}(\alpha) \neq 0 \), then \( p(D) \) is GH-\( 0 \) and \( \operatorname{ind}_{GH}(p(D)) = 0 \).
			\item[{\it (ii)}] If  \( \alpha \in \mathbb{Q} \), then \( p(D) \) is not GH-\( r \) for any \( r \geq 0 \).
			\item[{\it (iii)}] If \( \alpha \in \mathbb{R}\setminus\mathbb{Q} \) is an algebraic number of degree \( d \), then \( p(D) \) is GH-\( d \) and \( \operatorname{ind}_{GH}(p(D)) \leq d \).
		\end{enumerate}
        
	\begin{example}
		Consider the following vector field on the two-dimensional torus:
		\begin{equation*}
			p(D) = \partial_{x_1} - \sqrt[d]{\lambda} \, \partial_{x_2},
		\end{equation*}
		where \( d, \lambda \in \mathbb{N} \), \( d \geq 2 \), and \( \lambda \) is a prime number.
		
		Since \( \sqrt[d]{\lambda} \) is an algebraic irrational number of degree \( d \), it follows from the result above that \( p(D) \) is globally hypoelliptic with loss of \( d \) derivatives. This implies that $\operatorname{ind}_{GH}(p(D))\leq d$. 
	\end{example}
	
	Despite the aesthetic appeal of relating the degree of an algebraic irrational number to the corresponding order of global regularization, this result can be significantly improved by invoking Dirichlet's approximation theorem (see \cite{Schmidt_Diophantine_Book}) and the following version of the celebrated Thue--Siegel--Roth approximation theorem (see \cite{Roth1955}).
	
	\begin{theorem}[Roth, 1955]\label{Roth-Thm}
		Let \( \alpha \) be an algebraic irrational number and \( \varepsilon > 0 \). Then there exists a positive constant \( A = A(\alpha, \varepsilon) \) such that
		\begin{equation}\label{roth}
			\left| \alpha - \frac{\xi_1}{\xi_2} \right| \geq \frac{A}{|\xi_2|^{2+\varepsilon}}
		\end{equation}
		for all \( \xi_1, \xi_2 \in \mathbb{Z} \) with \( \xi_2 \neq 0 \).
	\end{theorem}
	
	As a consequence of \eqref{symbol-vf} and \eqref{roth}, for any given \( \varepsilon > 0 \), we obtain
	\begin{equation*}
		|p(\xi)| \geq \frac{A}{|\xi_2|^{1+\varepsilon}} \geq A |\xi|^{1-(2+\varepsilon)},
	\end{equation*}
	for every \( \xi = (\xi_1, \xi_2) \in \mathbb{Z}^2 \) such that \( \xi_2 \neq 0 \). Moreover, for \( \xi_2 = 0 \) and \( \xi_1 \in \mathbb{Z} \setminus \{0\} \), we have
	\begin{equation*}
		|p(\xi)| = |\xi_1| \geq |\xi|^{1-(2+\varepsilon)}.
	\end{equation*}
	
	It follows from Theorem \ref{theoGH} that \( p(D) = \partial_{x_1} - \alpha \partial_{x_2} \) is GH-\((2+\varepsilon)\) for every \( \varepsilon > 0 \).
	
	On the other hand, by Dirichlet's approximation theorem, for any irrational number \( \alpha \), there exist infinitely many reduced fractions \( \xi_1 / \xi_2 \) such that
	\begin{equation}\label{rothpat2}
		\left| \alpha - \frac{\xi_1}{\xi_2} \right| < \frac{1}{|\xi_2|^2}.
	\end{equation}
	
	Notice that this implies \( |\xi_2| > |\xi_1|/(\alpha+1) \). Therefore, for any given \( \varepsilon > 0 \), we have
	\begin{equation*}
		|p(\xi)| < \frac{1}{|\xi_2|} \leq \frac{1}{2|\xi_2|} \leq (|\alpha| + 1)(|\xi_1| + |\xi_2|)^{1-2}
	\end{equation*}
	for infinitely many pairs \( (\xi_1, \xi_2) \in \mathbb{Z}^2 \).
	
	It follows from Corollary \ref{coroGH} that \( p(D) = \partial_{x_1} - \alpha \partial_{x_2} \), with \( \alpha \) irrational, is not GH-\( r \) for any \( r < 2 \). Consequently,
	\[
	\operatorname{ind}_{GH}(p(D)) = 2.
	\]
	
	Thus, we have proved the following proposition.
	
	\begin{proposition}\label{Thm2-GHr-vf-T2}
		Let \( p(D) = \partial_{x_1} - \alpha \partial_{x_2} \) be a vector field on the two-dimensional torus, where \( \alpha \in \mathbb{R} \). The following statements hold:
		\begin{enumerate}
			\item[{\it (i)}] If \( \alpha \) is an algebraic irrational number then 
			\( \operatorname{ind}_{GH}(p(D)) = 2 \). 
            In particular, if $\alpha$ has degree \( 2 \), then \( p(D) \) is GH-\( 2 \).
			\item[{\it (ii)}] If \( \alpha \) is an irrational number, then \( p(D) \) cannot be GH-\( r \) for any \( r < 2 \), and $\operatorname{ind}_{GH}(p(D))\geq 2$.
		\end{enumerate}
	\end{proposition}
        
	It remains to analyse the case where the irrational number  \( \alpha \) is transcendental, meaning it is not a root of any polynomial with integer coefficients. We begin by recalling the definition of a Liouville number, which plays a crucial role in the study of global hypoellipticity.
	
	\begin{definition}\label{Liouv-number}
		An irrational number \( \alpha \) is said to be a Liouville number if there exists a sequence \( (k_j, \ell_j) \in \mathbb{Z}^2 \) such that \( \ell_j \to \infty \) and
		\begin{equation*}
			\left| \alpha - \frac{k_j}{\ell_j} \right| < \frac{1}{\ell_j^j}, \quad j \in \mathbb{N}.
		\end{equation*}
	\end{definition}
	
	In \cite{GW1972_pams}, it was established that the vector field \( p(D) = \partial_{x_1} - \alpha \partial_{x_2} \), with \( \alpha \in \mathbb{R} \setminus \mathbb{Q} \), is globally hypoelliptic if and only if \( \alpha \) is an irrational non-Liouville number. Consequently, by Corollary \ref{coro_GH_equiv}, if \( \alpha \) is a Liouville number, then \( p(D) \) fails to be GH-\( r \) for any \( r \geq 0 \). This implies that:
	\[
	\operatorname{ind}_{GH}(p(D)) = \infty.
	\]

	To complete our characterization of global hypoellipticity with loss of derivatives for vector fields on \( \mathbb{T}^2 \), we introduce the notion of \textit{irrationality measure} (also referred to as the \textit{irrationality exponent} in \cite[Appendix E]{Bugeaud_2012}). This concept quantifies how well a real number can be approximated by rational numbers.
	
	\begin{definition}
		The \textit{irrationality measure} \( \mu(\alpha) \) of a real number \( \alpha \) is defined as the infimum of all real numbers \( \mu > 0 \) such that the inequality
		\[
		0 < \left| \alpha - \frac{\xi_1}{\xi_2} \right| < \frac{1}{|\xi_2|^\mu},
		\]
		has at most finitely many solutions \( \xi_1 / \xi_2 \), where \( \xi_1, \xi_2 \in \mathbb{Z} \) and \( \xi_2 \neq 0 \).\\
		If the inequality has infinitely many solutions for all \( \mu > 0 \), then we set \( \mu(\alpha) = \infty \).
	\end{definition}

	\begin{remark}\label{measure-remark}
		For any real number \( \alpha \), the irrationality measure satisfies \( \mu(\alpha) = 1 \) when \( \alpha \) is rational, and \( \mu(\alpha) \geq 2 \) otherwise. More specifically, if \( \alpha \) is an algebraic irrational number, Roth's theorem (Theorem \ref{Roth-Thm}) implies that \( \mu(\alpha) = 2 \). For transcendental numbers, the irrationality measure can exceed 2; for instance, if \( \alpha \) is a Liouville number, then \( \mu(\alpha) = \infty \), as defined in Definition \ref{Liouv-number}. The base of the natural logarithm \( e \) satisfies \( \mu(e) = 2 \), while for \( \pi \), it is known that \( 2 \leq \mu(\pi) \leq 7.6063 \), although the exact value remains unknown. Proofs of these statements can be found in \cite{Bugeaud_2012} and \cite{pi}.
	\end{remark}
	
	A direct consequence of the definition of irrationality measure is that a real number \( \alpha \) has a finite irrationality measure \( \mu > 0 \) if and only if \( \mu \) is the smallest number such that for every \( \varepsilon > 0 \), there exists a constant \( A_\varepsilon > 0 \) satisfying
	\begin{equation*}
		\left| \alpha - \frac{\xi_1}{\xi_2} \right| \geq \frac{A_\varepsilon}{|\xi_2|^{\mu + \varepsilon}},
	\end{equation*}
	for all \( \xi_1, \xi_2 \in \mathbb{Z} \) with \( \xi_2 \neq 0 \).
	
	Using this characterization and following a similar reasoning as in Proposition \ref{Thm2-GHr-vf-T2}, we establish the following result:
	
	\begin{proposition}\label{Thm-irrational-measure}
		Let \( p(D) = \partial_{x_1} - \alpha \partial_{x_2} \) be a vector field on the two-dimensional torus, where \( \alpha \in \mathbb{R} \). If \( \alpha \) is irrational, then:
		\[
		\operatorname{ind}_{GH}(p(D)) = \mu(\alpha).
		\]
	\end{proposition}
	
	Another interesting result concerning global hypoellipticity with loss of derivatives for vector fields on the two-dimensional torus is as follows:
	
	\begin{proposition}\label{Lebesgue-measure}
		Let \( p(D) = \partial_{x_1} - \alpha \partial_{x_2} \), where \( \alpha \in \mathbb{R} \), be a vector field on the two-dimensional torus. Then:
		\[
		\operatorname{ind}_{GH}(p(D)) = 2 \quad \text{for almost every } \alpha \in \mathbb{R},
		\]
		with respect to the Lebesgue measure.
	\end{proposition}
		
	The proof of this result relies on showing that the set of real numbers \( \alpha \) satisfying \( \mu(\alpha) \neq 2 \) has Lebesgue measure zero. This fact was obtained by Khinchin in \cite[Theorem 32]{Khinchin} in a broader context. An elementary proof of this result, based on the Borel-Cantelli lemma and techniques from rational approximations in number theory, can be found in \cite[Appendix E]{Bugeaud_2012}.

	We conclude resuming the main  results of this section, providing a comprehensive characterization of global hypoellipticity and solvability with loss of derivatives for vector fields on the two-dimensional torus $\mathbb{T}^2$.	
	
	\begin{theorem}\label{Thm-GHr-GSr-characterization}
		Let \( p(D) = \partial_{x_1} - \alpha \partial_{x_2} \) be a vector field on \( \mathbb{T}^2 \), where \( \alpha \in \mathbb{C} \). Then:
		\begin{enumerate}
			\item[{\it (i)}] If \( \operatorname{Im}(\alpha) \neq 0 \), then:
			\[
			\operatorname{ind}_{GH}(p(D)) = \operatorname{ind}_{GS}(p(D)) = 0.
			\]
			Moreover, \( p(D) \) is both GH-\( 0 \) and GS-\( 0 \).
			
			\item[{\it (ii)}] If \( \alpha \in \mathbb{Q} \), then:
			\[
			\operatorname{ind}_{GH}(p(D)) = \infty \quad \text{and} \quad \operatorname{ind}_{GS}(p(D)) = 1.
			\]
			Moreover, \( p(D) \) is GS-\( 1 \).
			
			\item[{\it (iii)}] If \( \alpha \in \mathbb{R} \setminus \mathbb{Q} \), then:
			\[
			\operatorname{ind}_{GH}(p(D)) = \operatorname{ind}_{GS}(p(D)) = \mu(\alpha),
			\]
			Moreover, if \( \alpha \) is an algebraic number of degree \( 2 \), then \( p(D) \) is both GH-\( 2 \) and GS-\( 2 \).
		\end{enumerate}
	\end{theorem}

	\begin{proof}
		Items {\it (i)} and {\it (iii)}, as well as the global hypoellipticity with loss of derivatives in {\it (iii)}, follow directly from Corollary \ref{coroGHimpliesGS} and Propositions \ref{Thm2-GHr-vf-T2} and \ref{Thm-irrational-measure}.
		
		It remains to analyze the global solvability with loss of derivatives in the case where \( p(D) = \partial_{x_1} - \alpha \partial_{x_2} \) and \( \alpha = a/b \) is rational. Using Theorem \ref{theoGSindex}, we establish the global solvability index as follows.
				
		In this case, the symbol of the operator is given by:
		\[
		p(\xi) = i\left(\xi_1 - \frac{a}{b} \xi_2\right),
		\]
		and \( p(\xi) \neq 0 \) except for points \( (\xi_1, \xi_2) \in \mathbb{Z}^2 \) satisfying \( b \xi_1 = a \xi_2 \).
		
		For every \( \xi \in \mathbb{Z}^2 \setminus \{0\} \) with \( p(\xi) \neq 0 \), we have:
		\[
		|p(\xi)| = \left| \xi_1 - \frac{a}{b} \xi_2 \right| = \frac{1}{|b|} |b \xi_1 - a \xi_2| \geq \frac{1}{|b|} \geq \frac{1}{|b|} |\xi|^{1 - 1},
		\]
		and therefore by Theorem \ref{theoGS} we conclude that $p(D)$ is GS-$1$.
		
		On the other hand, consider the infinitely many points of the form \( \xi = ((1 + a \xi_2), b\xi_2) \in \mathbb{Z}^2 \). For all such $\xi\in\Z^2$, we have that
		\[
		 |p(\xi)| = 1 \leq  |\xi|^{1 - 1}.
		\]

		Therefore, by Corollary  \ref{coroGS1}, we conclude also that $\operatorname{ind}_{GS}(p(D))=1$.
	\end{proof}
	
	\begin{example}
		Let us examine the indices $\operatorname{ind}_{GH}(p(D))$ and $\operatorname{ind}_{GS}(p(D))$ for 
		\(p(D) = \partial_{x_1} - \alpha \partial_{x_2},\) for specific choices of the real coefficient $\alpha$.
		
		\paragraph{\bf 1. Euler's Number ($e$):} 
		It is well known that the irrationality measure of Euler's number is $\mu(e) = 2$. Therefore:
		\[
		\operatorname{ind}_{GH}(p(D)) = \operatorname{ind}_{GS}(p(D)) = \mu(e) = 2.
		\]
		
		\paragraph{\bf 2. Champernowne Constants ($C_b$):} 
		Masaaki Amou \cite{AMOU1991231} proved that the irrationality measure of the Champernowne constant $C_b$ in base $b \geq 2$ is exactly $b$. Consequently, for $\alpha = C_b$, we have:
		\[
		\operatorname{ind}_{GH}(p(D)) = \operatorname{ind}_{GS}(p(D)) = b.
		\]
		
		\paragraph{\bf 3. Gamma Function ($\Gamma(1/4)$):} 
		Michel Waldschmidt \cite{Waldschmidt2008} established an upper bound for the irrationality measure of $\Gamma(1/4)$, showing that $\mu(\Gamma(1/4)) \leq 10^{330}$. Thus, if $\alpha = \Gamma(1/4)$, the indices satisfy:
		\[
		2 \leq \operatorname{ind}_{GH}(p(D)) = \operatorname{ind}_{GS}(p(D)) \leq 10^{330}.
		\]
		In particular, if $\partial_{x_1}u - \Gamma(1/4) \partial_{x_2}u = f \in H^k(\mathbb{T}^2)$ for some $k \in \mathbb{R}$, then these estimates imply that $u \in H^{k + 1 - (10^{330} + \varepsilon)}$ for every $\varepsilon > 0$.
	\end{example}

\section{Application: The Wave Operator}
	
	In this section, we apply the results developed in the previous sections to study the loss of derivatives for the wave operator on the torus. We begin by analyzing the case of the two-dimensional torus, where we determine the precise values of the global hypoellipticity and global solvability indices. 
	
	Next, we extend our analysis to higher dimensions, establishing both lower and upper bounds for these indices. 
	
	\medskip 
\subsection{The wave operator on $\T^2$} \

	Consider the wave operator on the two-dimensional torus:
	\[
	p(D) = \partial_{x_1}^2 - \eta^2 \partial_{x_2}^2,
	\]
	where $\eta > 0$. Its symbol is given by:
	\[
	p(\xi) = -\xi_1^2 + \eta^2 \xi_2^2, \text{ with } \xi = (\xi_1, \xi_2) \in \mathbb{Z}^2.
	\]
	
	Clearly, $p(D)$ has intrinsic order $2$. If $\eta$ is irrational, the symbol $p(\xi)$ vanishes only when $\xi = 0$. Moreover, we can rewrite $|p(\xi)|$ as:
	\[
	|p(\xi)| = |-\xi_1^2 + \eta^2 \xi_2^2| = \big||\xi_1| - \eta |\xi_2|\big| \cdot \big||\xi_1| + \eta |\xi_2|\big|.
	\]
	
	Now, if $\eta$ is an irrational algebraic number, Theorem \ref{Roth-Thm} implies that for every $\varepsilon > 0$, there exists a constant $A = A(\eta, \varepsilon) > 0$ such that:
	\[
	\left|\frac{\xi_1}{\xi_2} - \eta\right| \geq \frac{A}{|\xi_2|^{2+\varepsilon}}, \quad \text{for all } \xi_2 \neq 0.
	\]
	Using this inequality, we estimate $|p(\xi)|$ for $\xi_2 \neq 0$:
	\begin{align*}
		|p(\xi)| &= |\xi_2| \cdot \left|\frac{\xi_1}{\xi_2} - \eta\right| \cdot \left||\xi_1| + \eta |\xi_2|\right| \\
		&\geq |\xi_2| \cdot \frac{A}{|\xi_2|^{2+\varepsilon}} \cdot \min\{1, \eta\} |\xi| \\
		&= A \min\{1, \eta\} |\xi|^{2 - (2 + \varepsilon)}.
	\end{align*}
	For the case $\xi_2 = 0$, we have:
	\[
	|p(\xi)| = |\xi_1|^2 = |\xi|^2 \geq |\xi|^{2 - (2 + \varepsilon)}.
	\]
	
	Thus, combining both cases, we conclude that for every $\varepsilon > 0$, there exists a constant $C > 0$ such that:
	\[
	|p(\xi)| \geq C |\xi|^{2 - (2 + \varepsilon)},
	\]
	for all $\xi \in \mathbb{Z}^2 \setminus \{0\}$. By Theorem \ref{theoGH}, it follows that $p(D)$ is GH-$r$ for every $r > 2$, and hence:
	\[
	\operatorname{ind}_{GH}(p(D)) \leq 2.
	\]
	
	On the other hand, by Dirichlet's approximation theorem, there exist infinitely many rational approximations $\xi_1 / \xi_2 \in \mathbb{Q}$ such that:
	\[
	\left|\frac{\xi_1}{\xi_2} - \eta\right| < \frac{1}{|\xi_2|^2}.
	\]
	Using this, we estimate $|p(\xi)|$ for such frequencies:
	\begin{align*}
		0 < |p(\xi)| &= |\xi_2|^2 \cdot \left|\frac{\xi_1}{\xi_2} - \eta\right| \cdot \left|\frac{\xi_1}{\xi_2} + \eta\right| \\
		&< |\xi_2|^2 \cdot \frac{1}{|\xi_2|^2} \cdot (2\eta + 1) 
		= (2\eta + 1) |\xi|^{2 - 2}.
	\end{align*}
	This shows that $p(D)$ is not GH-$r$ for any $r < 2$. Combining this with the lower bound, we conclude:
	\[
	\operatorname{ind}_{GH}(p(D)) = 2.
	\]
	
	If $\eta$ is transcendental, similar arguments to those in the previous section, along with a repetition of the ideas above imply that
	\[
	\operatorname{ind}_{GH}(p(D)) = \mu(\eta),
	\]
	where $\mu(\eta)$ denotes the irrationality measure of $\eta$.

	Next, consider the case where $\eta = a/b \in \mathbb{Q}_+$ (with $a, b \in \mathbb{N}$). The symbol of the operator is given by:
	\[
	|p(\xi)| = |-\xi_1^2 + (a/b)^2 \xi_2^2| = \frac{1}{b} \big||b\xi_1| - a|\xi_2|\big| \cdot \big||\xi_1| + (a/b)|\xi_2|\big|.
	\]
	
	For every $\xi \in \mathbb{Z}^2$ such that $p(\xi) \neq 0$, we have $b|\xi_1| - a|\xi_2| \in \mathbb{Z} \setminus \{0\}$, and thus:
	\[
	|p(\xi)| \geq \frac{1}{b} \min\{1, a/b\} |\xi|^{2-1}.
	\]
	This implies that $p(D)$ is GS-$1$.
	
	To verify that $\operatorname{ind}_{GS}(p(D)) = 1$, consider the sequence of integer pairs:
	\[
	\xi_{1,j} = a(j+1), \quad \xi_{2,j} = bj, \quad \text{for } j \in \mathbb{N}.
	\]
	
	For this sequence, we have:
	\[
	|p(\xi_{1,j}, \xi_{2,j})| = |-a^2(j+1)^2 + a^2j^2| = a(2j+1).
	\]
	
	Moreover, observe that:
	\[
	|\xi_{1,j}| + |\xi_{2,j}| = (a+b)j + a \geq 2j+1, \quad \text{for all } j \in \mathbb{N}.
	\]
	
	Thus, we have:
	\[
	0 < |p(\xi_{1,j}, \xi_{2,j})| \leq a(|\xi_{1,j}| + |\xi_{2,j}|)^{2-1}, \quad \text{for all } j \in \mathbb{N}.
	\]
	
	By Theorem \ref{theoGSindex}, it follows that 
	\[\operatorname{ind}_{GS}(p(D)) = 1.\]
	
	In conclusion, we obtain the following complete characterization of the hypoellipticity and solvability indices  for the wave operator on $\mathbb{T}^2$.

	\begin{theorem}\label{thm-wave-n=2}
		Let $\eta > 0$. The wave operator $p(D) = \partial_{x_1}^2 - \eta^2 \partial_{x_2}^2$ on $\mathbb{T}^2$ satisfies the following properties:
		\begin{itemize}
			\item[{\it (i)}] If $\eta$ is irrational, then $\operatorname{ind}_{GH}(p(D)) = \mu(\eta)$.
			
			\item[{\it (ii)}] If $\eta \in \mathbb{Q}$, then $\operatorname{ind}_{GH}(p(D)) = \infty$.
			
			\item[{\it (iii)}] $\operatorname{ind}_{GS}(p(D)) = \mu(\eta)$. Moreover, if $\eta \in \mathbb{Q}$, then $p(D)$ is GS-$1$.
			
			\item[{\it (iv)}] If $\eta$ is an algebraic number of degree $2$, then $p(D)$ is GH-$2$ and GS-$2$.
		\end{itemize}
	\end{theorem}

\subsection{Wave operator on $\T^{n+1}$} \

	We now turn our attention to the wave operator on the $(n+1)$-dimensional torus $\mathbb{T}^1 \times \mathbb{T}^n$, where $n \in \mathbb{N}$. The operator is given by:
	\[
	p(D) = \partial_{x_1}^2 - \eta^2 \Delta_{\mathbb{T}^n}, 
	\]
	where $\eta > 0$. Its symbol is:
	\[
	p(\xi) = -\xi_1^2 + \eta^2 \|\xi'\|^2, \quad \xi = (\xi_1, \xi') \in \mathbb{Z}^{n+1}.
	\]
	As Before, $p(D)$ has (intrinsic) order $2$.
	
	Unlike the two-dimensional case, the analysis here requires additional care because the square root of the symbol of $\Delta_{\mathbb{T}^n}$ is not always an integer. Consequently, the techniques used in the two-dimensional setting cannot be applied directly. Nevertheless, we are still able to adapt similar ideas to derive lower and upper bounds for the global hypoellipticity and solvability indices, as detailed below.

	First, consider the case where $\eta^2$ is irrational. In this case, the symbol $p(\xi)$ vanishes only when $\xi = 0$. 
	
	Additionally, observe that:
	\[
	|p(\xi)| = \|\xi'\|^2 \left| \frac{\xi_1^2}{\|\xi'\|^2} - \eta^2 \right|.
	\]
	For $\xi' \neq 0$, we can apply the definition of the irrationality measure $\mu(\eta^2)$ to estimate the term $\left| \frac{\xi_1^2}{\|\xi'\|^2} - \eta^2 \right|$. Specifically, for every $\varepsilon > 0$, there exists a constant $A_\varepsilon > 0$ such that:
	\[
	\left| \frac{\xi_1^2}{\|\xi'\|^2} - \eta^2 \right| \geq \frac{A_\varepsilon}{\|\xi'\|^{2(\mu(\eta^2) + \varepsilon)}}.
	\]
	Using this inequality, we obtain:
	\begin{align*}
		|p(\xi)| &= \|\xi'\|^2 \left| \frac{\xi_1^2}{\|\xi'\|^2} - \eta^2 \right| 
		\geq \|\xi'\|^2 \cdot \frac{A_\varepsilon}{\|\xi'\|^{2(\mu(\eta^2) + \varepsilon)}} \\
		&= \frac{A_\varepsilon}{\|\xi'\|^{2\mu(\eta^2) - 2 + 2\varepsilon}} 
		\geq A'_\varepsilon |\xi|^{2 - (2\mu(\eta^2) + 2\varepsilon)},
	\end{align*}
	for all $\xi \in \mathbb{Z}^{n+1} \setminus \{0\}$ such that $\xi' \neq 0$, where $A'_\varepsilon > 0$ is a constant depending on $\varepsilon$.
	
	When $\xi' = 0$, we have:
	\[
	|p(\xi)| = |\xi_1|^2 = |\xi|^2 \geq |\xi|^{2 - (2\mu(\eta^2) + 2\varepsilon)},
	\]
	for all $\xi_1 \in \mathbb{Z}$. Combining both cases, it follows that:
	\[
	|p(\xi)| \geq A''_\varepsilon |\xi|^{2 - (2\mu(\eta^2) + 2\varepsilon)},
	\]
	for all $\xi \in \mathbb{Z}^{n+1} \setminus \{0\}$, where $A''_\varepsilon > 0$ is another constant depending on $\varepsilon$.

	Moreover, note that by our assumption, $\eta$ is also irrational. Therefore, by definition, given $\mu(\eta)\geq\varepsilon>0$ there exist infinitely many approximations $\xi_1/\xi_2\in\mathbb{Q}$ such that:
	\[
	\left| \frac{\xi_1}{\xi_2} - \eta \right| < \frac{1}{|\xi_2|^{\mu(\eta)-\varepsilon}}.
	\]

	For these approximations and $\xi = (\xi_1, \xi_2, 0, \dots, 0) \in \mathbb{Z}^{n+1}$, we have:
	\begin{align*}
		0 & < |p(\xi_1, \xi_2, 0, \dots, 0)| = |\xi_2|^2 \left| \frac{\xi_1}{\xi_2} - \eta \right| \left| \frac{\xi_1}{\xi_2} + \eta \right| \\
		&< |\xi_2|^2 \cdot \frac{1}{|\xi_2|^{\mu(\eta)-\varepsilon}} \cdot (2\eta + 1) 
		< (2\eta + 1)|\xi|^{2-(\mu(\eta)-\varepsilon)},
	\end{align*}
	where in the second inequality we used the fact that $\left| \xi_1/\xi_2 - \eta \right| < 1/|\xi_2|^{\mu(\eta)-\varepsilon}$ implies $\left| \xi_1/\xi_2 + \eta \right| \leq 2\eta + 1$.

    For these $\xi$, we conclude that:
	\[
	0 < |p(\xi)| < (2\eta + 1) |\xi|^{2-(\mu(\eta)-\varepsilon)}.
	\]
	
	Thus, by Theorem \ref{theoGHindex}, we conclude that:
	\[
	\mu(\eta) \leq \operatorname{ind}_{GH}(p(D)) \leq 2\mu(\eta^2).
	\]
    Note that this aligns with the fact that $\mu(\alpha^{1/2})\leq 2\mu(\alpha)$ for any irrational number $\alpha>0$ (\cite{Borwein}).
    
	Now assume that $\eta$ is irrational, but $\eta^2$ is rational. Then whether or not $p(D)$ is GH-$r$ for some $r \geq 0$ depends on the dimension $n$ and on the constant $\eta$. 
	
	Indeed, write $\eta = \sqrt{a/b}$, where $a, b \in \mathbb{N}$, $\gcd(a, b) = 1$. Then we consider the cases $n \geq 4$, $n = 3$, and $n = 2$ separately, as follows.
	
	\smallskip
	If $n \geq 4$, by Lagrange's four-square theorem, any non-negative integer can be written as the sum of four squares. For any $ab \xi_1 \in \mathbb{Z}$, choose $\xi_2, \dots, \xi_n \in \mathbb{Z}$ such that:
	\[
	\xi_2^2 + \dots + \xi_n^2 = ab \xi_1^2.
	\]
	Substituting into the symbol $p(\xi)$, we find:
	\[
	p(\xi) = -a^2 \xi_1^2 + \frac{a}{b} (\xi_2^2 + \dots + \xi_n^2) = -a^2 \xi_1^2 + \frac{a}{b} (ab \xi_1^2) = 0.
	\]
	Thus, $p(\xi) = 0$ for infinitely many $\xi \in \mathbb{Z}^{n+1}$. By Theorem \ref{theoGH}, $p(D)$ is not GH-$r$ for any $r \geq 0$.
	
	\smallskip
	If $n = 3$, then notice that $p(\xi) = 0$ if and only if:
	\[
	-a^2 \xi_1^2 + \frac{a}{b} (\xi_2^2 + \xi_3^2 + \xi_4^2) = 0,
	\]
	which simplifies to:
	\[
	|\xi_1| = \sqrt{\frac{a}{b} (\xi_2^2 + \xi_3^2 + \xi_4^2)}.
	\]
	
	For this equality to hold, $\frac{a}{b} (\xi_2^2 + \xi_3^2 + \xi_4^2)$ must be a perfect square. 
	
 If $\xi\neq 0$, writing $a = c^2 d$, where $d$ is square-free, we see that $(\xi_2, \xi_3, \xi_4) \in \mathbb{Z}^3$ must satisfy:
	\[
	\xi_2^2 + \xi_3^2 + \xi_4^2 = b d j^2, \quad \text{for some } j \in \mathbb{N}.
	\]
		
	By Legendre's three-square theorem, this is possible if and only if $b d j^2$ is not of the form $4^{\ell_1} (8 \ell_2 + 7)$, where $\ell_1, \ell_2 \in \mathbb{N}$. 
	
	We claim that there exist infinitely many $j \in \mathbb{N}$ such that $b d j^2$ satisfies this condition.
	
	To analyze this, write $b d = 4^\ell \lambda$, where $\lambda \in \mathbb{N}$ and $4 \nmid \lambda$. Then we have $b d j^2 = 4^{\ell_1} (8 \ell_2 + 7)$ for some $\ell_1, \ell_2 \in \mathbb{N}$ if and only if:
	\[
	\lambda j^2 = 4^{\ell'_1} (8 \ell_2 + 7), \quad \text{for some } \ell'_1, \ell_2 \in \mathbb{N}.
	\]

	Now, choose $j$ of the form $j = 2(2k + 1)$, where $k \in \mathbb{N}$. For this choice of $j$, the equality above is only possible if:
	\[
	\lambda j^2 = 4 (8 \ell_2 + 7), \quad \text{for some } \ell_2 \in \mathbb{N},
	\]
	because $4 \nmid \lambda$, $4\mid j^2$ and $4^2 \nmid j^2$. Notice that this implies:
	\[
	\lambda j^2 \equiv 3 \pmod{8}.
	\]
	However, since $j^2 \equiv 4 \pmod{8}$, we have:
	\[
	\lambda j^2 \equiv 0 \text{ or } 4 \pmod{8},
	\]
	a contradiction. 
	
	Thus, $b d j^2$ is not of the form $4^{\ell_1} (8 \ell_2 + 7)$, where $\ell_1, \ell_2 \in \mathbb{N}$.
	
	Consequently, there exist $\xi_2, \xi_3, \xi_4 \in \mathbb{Z}$ such that:
	\[
	\xi_2^2 + \xi_3^2 + \xi_4^2 = b d j^2,
	\]
	and hence:
	\[
	\sqrt{\frac{a}{b} (\xi_2^2 + \xi_3^2 + \xi_4^2)} \quad \text{is an integer.}
	\]
	
	Choosing $|\xi_1| = ( a/b (\xi_2^2 + \xi_3^2 + \xi_4^2))^{1/2}$, we obtain $\xi \in \mathbb{Z}^4$ such that $p(\xi) = 0$. 
	
	Since there are infinitely many $j \in \mathbb{N}$ of the form $2(2k + 1)$, where $k \in \mathbb{N}$, it follows that $p(\xi) = 0$ for infinitely many $\xi$. By Theorem \ref{theoGH}, we conclude that $p(D)$ is not GH-$r$ for any $r \geq 0$.

	\smallskip
	If $n = 2$, then notice that $p(\xi) = -a^2 \xi_1^2 + \frac{a}{b} (\xi_2^2 + \xi_3^2) = 0$ if and only if:
	\[
	|\xi_1| = \sqrt{\frac{a}{b} (\xi_2^2 + \xi_3^2)}.
	\]
	
	This implies that $ \frac{a}{b} (\xi_2^2 + \xi_3^2)$ must be a perfect square. Writing $a = c^2 d$, where $d$ is square-free, we see that for $p(\xi)$ to vanish,  either $\xi=0$ or $(\xi_2, \xi_3) \in \mathbb{Z}^2$ must satisfy:
	\[
	\xi_2^2 + \xi_3^2 = b d j^2, \quad \text{for some } j \in \mathbb{Z}.
	\]
	
	By the sum of two squares theorem, this is possible if and only if $b d j^2$ contains no prime factor $q^k$, where $q \equiv 3 \pmod{4}$ and $k$ is odd. Since $j^2$ can only contain even powers of primes, and $\gcd(b, d) = 1$, it follows that if neither $b$ nor $d$ (and consequently $a$) contains such a prime factor, then $b d j^2$ also contains no such prime factor.
	
	In conclusion, if neither $a$ nor $b$ contains a prime factor $q^k$, where $q \equiv 3 \pmod{4}$ and $k$ is odd, then $p(D)$ is not GH-$r$ for any $r \geq 0$, since its symbol has infinitely many zeros.

	On the other hand, if either $a$ or $b$ contains a prime factor $q^k$, where $q \equiv 3 \pmod{4}$ and $k$ is odd, then $b d j^2$ also contains such a prime factor. This follows because $\gcd(a, b) = 1$, so $a$ and $b$ share no common prime factors, and since $d$ is the square-free part of $a$, any odd-powered prime factor $q^k$ of $a$ will appear as a factor $q$ in $d$. 
	
	Consequently, $b d j^2$ cannot be written as a sum of two squares for any $j \in \mathbb{N}$. Thus, $p(\xi) = 0$ only if $\xi = 0$.
	
	Moreover, 
	\begin{align*}
		|p(\xi)| &= \frac{1}{\sqrt{b}} \left| \sqrt{b |\xi_1|^2} - \sqrt{a (\xi_2^2 + \xi_3^2)} \right| \cdot \left| |\xi_1| + \sqrt{\frac{a}{b} (\xi_2^2 + \xi_3^2)} \right| \\
		&= \frac{1}{\sqrt{b}} \left| \frac{b |\xi_1|^2 - a (\xi_2^2 + \xi_3^2)}{\sqrt{b |\xi_1|^2} + \sqrt{a (\xi_2^2 + \xi_3^2)}} \right| \cdot \left| |\xi_1| + \sqrt{\frac{a}{b} (\xi_2^2 + \xi_3^2)} \right|.
	\end{align*}
	Using the inequality $\sqrt{a} + \sqrt{b} \leq \sqrt{2} \sqrt{a + b}$ (valid for $a, b > 0$) and the fact that $b |\xi_1|^2 - a (\xi_2^2 + \xi_3^2) \in \mathbb{Z}$, we obtain:
	\begin{align*}
		|p(\xi)| &\geq \frac{1}{\sqrt{b}} \cdot \frac{1}{\sqrt{2 \min\{a, b\}} \sqrt{|\xi_1|^2 + |\xi_2|^2 + |\xi_3|^2}} \cdot \min\left\{1, \frac{\sqrt{a/b}}{2}\right\} |\xi| \\
		&\geq K |\xi|^{2-2},
	\end{align*}
	for all $(\xi_1, \xi_2, \xi_3) \in \mathbb{Z}^3 \setminus \{0\}$, where $K > 0$ depends only on $a$ and $b$.
	
	By Theorem \ref{theoGH}, it follows that $p(D)$ is GH-$2$. Since $\eta$ is irrational and $\mu(\eta)=2$ (since it is algebraic), we can repeat the argument from the case where $\eta^2$ was irrational to conclude that $p(D)$ is not GH-$r$ for any $r < 2$. Therefore, in this case:
	\[
	\mu(\eta)=\operatorname{ind}_{GH}(p(D)) = \mu(\eta^2)=2.
	\]

	Notice that the arguments above can also be applied to the cases $n \geq 4$ and $n = 3$. Specifically, whenever $p(\xi) \neq 0$, we have:
	\[
	|p(\xi)| \geq K |\xi|^{2-2}, \quad \text{for some } K > 0.
	\]
	
	Thus, by Theorem \ref{theoGS}, it follows that $p(D)$ is GS-$2$. Moreover, the same sequence used to prove that $p(D)$ is not GH-$r$ for any $r < 2$ can also be employed to conclude that $p(D)$ is not GS-$r$ for any $r < 2$. Therefore:
	\[
	\operatorname{ind}_{GS}(p(D)) = 2, \quad \text{if } n \geq 4 \text{ or } n = 3.
	\]
	
	Finally, let us consider the case where $\eta = a/b \in \mathbb{Q}_+$, with $a, b \geq 1$. Then:
	\[
	|p(\xi)| = \frac{1}{b^2} |b^2 \xi_1^2 - a^2 \|\xi'\|^2| \geq \frac{1}{b^2} = \frac{1}{b^2} |\xi|^{2-2},
	\]
	whenever $p(\xi) \neq 0$, since $|b^2 \xi_1^2 - a^2 \|\xi'\|^2|$ is an integer for every $\xi \in \mathbb{Z}^{n+1}$. By Theorem \ref{theoGS}, we conclude that $p(D)$ is GS-$2$.
	
	On the other hand, consider the sequence of vectors $(\xi_{1, j},\ldots, \xi_{n+1, j} )_{j \in \mathbb{N}}$ in \( \mathbb{Z}^{n+1}\) given by:
	\[
	\begin{cases}
		\xi_{1,j} = a(j+1), \\
		\xi_{2,j} = bj, \\
		\xi_{\ell,j} = 0, \text{ for } 3 \leq \ell \leq n+1.
	\end{cases}
	\]
	for every $j \in \mathbb{N}$. For this sequence, we have:
	\[
	|p(\xi_{1, j},\ldots, \xi_{n+1, j}) | = |-a^2 (j+1)^2 + a^2 j^2| = a^2(2j+1),
	\]
	for every $j \in \mathbb{N}$. Moreover, observe that:
	\[
	|(\xi_{1, j},\ldots, \xi_{n+1, j})| = |\xi_{1,j}| + |\xi_{2,j}| = (a+b)j + a \geq 2j+1,
	\]
	for every $j \in \mathbb{N}$. 
	
	Hence:
	\[
	0 < |p(\xi_{1, j},\ldots, \xi_{n+1, j})| \leq a^2 |(\xi_{1, j},\ldots, \xi_{n+1, j})|^{2-1},
	\]
	for every $j \in \mathbb{N}$. By Theorem \ref{theoGSindex}, we conclude that:
	\[
	1 \leq \operatorname{ind}_{GS}(p(D)) \leq 2.
	\]

	As before, we conclude with the following complete characterization of the hypoellipticity and solvability indices for the wave operator on $\mathbb{T}^{n+1}$.

\begin{theorem}\label{theo-wave-n+1}
    Let $\eta >0$. The wave operator $p(D)=\partial_{x_1}^2-\eta^2\Delta_{\T^n}$ on $\T^1\times\T^n$ satisfies the following properties:
    \begin{itemize}
        \item[{\it (i)}] If $\eta,\eta^2\not\in\mathbb{Q}$, then 
        $\mu(\eta)\leq \operatorname{ind}_{GH}(p(D))\leq 2\mu(\eta^2)$. 
        \item[{\it (ii)}] If $\eta\not\in\mathbb{Q}$, but $\eta^2=\frac{a}{b}\in\mathbb{Q}$, $\gcd(a,b)=1$, then $\operatorname{ind}_{GS}(p(D))=2$ and: 
        \begin{itemize}
            \item If $n\geq 3$, then  $\operatorname{ind}_{GH}(p(D))=\infty$.
            \item If $n=2$ and either $a$ or $b$ contain a prime factor $q^k$, where $q\equiv 3 \pmod{4}$ and $k$ is odd, then $\operatorname{ind}_{GH}(p(D))=2$ and $p(D)$ is GH-$2$. Otherwise, $\operatorname{ind}_{GH}(p(D))=\infty$
            \item If $n=2$
        \end{itemize}
          \item[{\it (iii)}] If $\eta\in\mathbb{Q}$, then $\operatorname{ind}_{GH}(p(D))=\infty$ and $1\leq\operatorname{ind}_{GS}(p(D))\leq 2$.
    \end{itemize}
\end{theorem}

\begin{corollary}
     Let $\eta >0$. The wave operator $p(D)=\partial_{x_1}^2-\eta^2\Delta_{\T^n}$ on $\T^1\times\T^n$ satisfies  $\mu(\eta)\leq \operatorname{ind}_{GS}(p(D))\leq 2\mu(\eta^2)$, and it is not globally hypoelliptic if and only if $\eta$ satisfies any of the conditions below:
     \begin{itemize}
         \item[{\it (i)}] $\eta\in\mathbb{Q}$.
         \item[{\it (ii)}] $\eta\not\in\mathbb{Q}$, but $\eta^2$ and either $n\geq 3$ or $n=2$ and neither $a$ nor $b$ contain a prime factor $q^k$, where $q\equiv 3 \pmod 4$ and $k$ is odd.
     \end{itemize}
     Moreover, whenever $p(D)$ is globally hypoelliptic, it satisfies \[\mu(\eta)\leq \operatorname{ind}_{GH}(p(D))\leq 2\mu(\eta^2).\]
\end{corollary}
\begin{remark}
    Notice that Theorems \ref{thm-wave-n=2} and \ref{theo-wave-n+1} can be seen as an extension of a result by Bergamasco and Zani in \cite{bergamasco-second-order}, where they proved that for $\eta>0$, the wave operator $\partial_{x_1}-\eta^2\partial_{x_2}$ on $\T^2$ is globally hypoelliptic if and only if $\eta$ is neither rational nor a Liouville number.
\end{remark} 
\begin{example}
    Let $\eta>0$ and $\lambda>0$ is a prime number and $k\in\N_0$, $k>\frac{n+1}{2}$. Then the wave equation $\partial_{x_1}^2u-\lambda\Delta_{\T^n}u=f\in C^{k}(\T^{n+1})$, admits solution in $C^{k-\lceil \frac{n}{2}\rceil-1}(\T^{n+1})$ for every $f$ satisfying  $\widehat{f}(0)=0$. Moreover, every solution to such equation is in $C^{k-2}(\T^{n+1})$,  if and only if $n=1$ or $n=2$ and $\lambda\equiv3 \pmod 4$. Also, if $n\geq 3$ then the solution $u$ can can exhibit arbitrarily low regularity, despite of the smoothness of $f$.
\end{example}
\begin{example}
    Since $\mu(e)=\mu(e^2)=2$, Theorem \ref{theo-wave-n+1} implies that the wave operator $p(D)=\partial_{x_1}^2-e^2\Delta_{\T^n}$ on $\T^{1}\times\T^n$ satisfies  $2\leq \operatorname{ind}_{GS}(p(D))\leq 4$. Moreover, if $\eta$ is an irrational algebraic number and $\eta^2$ is also irrational, Theorem \ref{theo-wave-n+1} implies that $2\leq \operatorname{ind}_{GH}(p(D))\leq 4$, since $\eta^2$ is also algebraic.
\end{example}

\bibliographystyle{plain} 
\bibliography{references_final} 

\end{document}